\documentclass[12pt]{amsart}
\usepackage{amsmath}
\usepackage{amsfonts}
\usepackage{amsthm}
\usepackage{amssymb}
\usepackage{amscd}
\usepackage[all]{xy}
\usepackage{enumerate}

\keywords{Extension class of a vector bundle, torsion freeness, 
infinitesimal Torelli problem, canonical map, holomorphic forms,
Albanese variety, families of varieties, generic Torelli problem.} 

\subjclass[2010]{14C34, 14D07, 14E99, 14J10, 14J40.}

\theoremstyle{plain}
\newtheorem{thm}{Theorem}[subsection]
\newtheorem{thml}{Theorem}

\newtheorem{prop}[thm]{Proposition}

\newtheorem{cor}[thm]{Corollary}
\newtheorem{rem}[thm]{Remark}

\theoremstyle{definition}
\newtheorem{defn}[thm]{Definition}

\newtheorem{rmk}[thm]{Remark}
\newcommand{\sA}{\mathcal{A}}
\newcommand{\sB}{\mathcal{B}}

\newcommand{\sE}{\mathcal{E}}
\newcommand{\sF}{\mathcal{F}}

\newcommand{\sI}{\mathcal{I}}

\newcommand{\sK}{\mathcal{K}}
\newcommand{\sL}{\mathcal{L}}

\newcommand{\sM}{\mathcal{M}}
\newcommand{\sO}{\mathcal{O}}
\newcommand{\sP}{\mathcal{P}}
\newcommand{\sQ}{\mathcal{Q}}

\newcommand{\sX}{\mathcal{X}}
\newcommand{\sY}{\mathcal{Y}}

\newcommand{\mP}{\mathbb{P}}

\newcommand{\Ima}{\mathrm{Im}\,}
\newcommand{\Ker}{\mathrm{Ker}\,}

\newcommand{\Hom}{\mathrm{Hom}}

\numberwithin{equation}{section}

\newcommand{\beba}  {\begin{equation}\begin{array}{rcl}}

\newcommand{\eaee}  {\end{array}\end{equation}}

\makeatletter
\def\l@section{\@tocline{1}{0pt}{1pc}{}{}}
\def\l@subsection{\@tocline{2}{0pt}{1pc}{4.6em}{}}
\def\l@subsubsection{\@tocline{3}{0pt}{1pc}{7.6em}{}}
\renewcommand{\tocsection}[3]{%
  \indentlabel{\@ifnotempty{#2}{\makebox[2.3em][l]{%
    \ignorespaces#1 #2.\hfill}}}#3}
\renewcommand{\tocsubsection}[3]{%
  \indentlabel{\@ifnotempty{#2}{\hspace*{2.3em}\makebox[2.3em][l]{%
    \ignorespaces#1 #2.\hfill}}}#3}
\renewcommand{\tocsubsubsection}[3]{%
  \indentlabel{\@ifnotempty{#2}{\hspace*{4.6em}\makebox[3em][l]{%
    \ignorespaces#1 #2.\hfill}}}#3}
\makeatother

\title{Differential forms and quadrics of the canonical image}

\author{Luca Rizzi}
\address{D.I.M.I. \\
the University of Udine\\
Udine, 33100 Italy\\
\texttt{rizzi.luca@spes.uniud.it}}

\author{Francesco Zucconi}

\address{D.I.M.I. \\
the University of Udine\\
Udine, 33100 Italy\\
\texttt{Francesco.Zucconi@dimi.uniud.it}}

\begin{document}

\markboth{Rizzi and Zucconi}{Differential forms and quadrics of the canonical image}

\begin{abstract} We extend some of the results of \cite{PZ} and we 
    prove a criterion for a family $\pi\colon\sX\to B$ 
    of $n$-dimensional varieties of general type and with Albanese 
    morphism of degree $1$ to have birational fibers. 
    We do not assume that the canonical map of the general fiber is a 
    morphism nor that the canonical linear system has 
    no fixed components. We also prove some instances 
 of the generic Torelli theorem under the 
    assumption, natural in this context, that the fibers are minimal and their minimal model
    is unique. It 
    is trivial to construct counterexamples to generic Torelli 
    without such minimality assumptions. 
     \end{abstract}

\maketitle
\tableofcontents

\section{Introduction}
 Given an $n\times n$ matrix $T\in {\rm{Mat}}(n,\mathbb K)$, there 
 always exists its adjoint matrix, 
 $T^{\vee}$ such that by row-column product we obtain $T\cdot 
 T^{\vee}={\rm{det}}(T)I_{n}$, where $I_{n}\in {\rm{Mat}}(n,\mathbb 
 K)$ is the identity matrix. In this paper we consider the above construction in 
 the context of locally free sheaves.
 
 \subsection{General theory}
 Let $\xi\in 
 \text{Ext}^1(\mathcal{F},\mathcal{O}_X)$ be an extension class 
 associated to the following exact sequence of locally free sheaves 
 over an $m$-dimensional smooth variety $X$:
\begin{equation}\label{la sequenzasolita}
0\to\mathcal{O}_X\stackrel{}{\rightarrow} \mathcal{E}\stackrel{}{\rightarrow} \mathcal{F}\to 0.
\end{equation}
Assume that $\sF$ is of rank $n$ and that the kernel of the 
connecting homomorphism  $\partial_{\xi}\colon 
H^{0}(X,\sF)\to H^{1}(X,\sO_{X})$ has dimension $\geq n+1$. Take an
$n+1$-dimensional subspace $W\subset \Ker\partial_{\xi}$ and a basis $\sB=\{\eta_{1},\ldots,\eta_{n+1}\}$. By choosing a lifting 
$s_{i}\in H^{0}(X,\sE)$ of $\eta_{i}$, where $i=1,\ldots, 
n+1$, we have a top form $\Omega\in H^{0}(X,\det\sE)$ from the 
element $s_{1}\wedge\ldots\wedge s_{n+1}\in 
\bigwedge^{n+1}H^{0}(X,\sE)$. Since $\det\sE=\det\sF$ we 
actually obtain from $\Omega$ a top form $\omega$ of $\sF$, which depends on the chosen 
liftings and on $\sB$. We call such an $\omega\in 
H^{0}(X,\det\sF)$ an \emph{adjoint form} of $\xi, W, \sB$. Indeed consider the $n+1$ top forms 
$\omega_{1},\ldots ,\omega_{n+1}\in H^{0}(X,\det\sF)$ where 
$\omega_{i}$ is obtained by the element 
$\eta_{1}\wedge\ldots\wedge{\widehat{\eta_{i}}}\wedge\ldots\wedge 
\eta_{n+1}\in \bigwedge^{n}H^{0}(X,\sF)$. It is easily seen 
that the subscheme of $X$ where $\omega$ vanishes is locally given by 
the vanishing of the determinant of a suitable $(n+1)\times (n+1)$ 
matrix $T$ and the local expressions of $\omega_{i}$, $i=1,\ldots, 
n+1$, give some entries of $T^{\vee}$. This idea to relate 
extension classes to adjoint forms was first introduced in \cite{CP} 
for the case of smooth curves and later the theory was extended
to any smooth algebraic 
variety; see: \cite{PZ}. 
Since then it has been fruitfully applied in \cite{Ra}, \cite{PR}, \cite{CNP}, \cite{victor1}, \cite{victor2} and \cite{BGN}.

In this paper we go deeper along the direction indicated in \cite{PZ}.

To the $n+1$-dimensional subspace $W\subset \Ker\partial_{\xi}$ 
taken above, we associate $D_{W}$ and $Z_{W}$ which are 
respectively the fixed part and 
the base loci of the sublinear system given by $\lambda^{n} W \subset H^{0}(X,\det \sF)$, where $\lambda^{n}\colon 
\bigwedge^{n} H^{0}( X,\sF)\to H^{0}(X,\det\sF)$ is the natural 
homomorphism obtained by the wedge product. We prove:

\begin{thml}
Let $X$ be an $m$-dimensional compact complex smooth 
variety. Let $\sF$ be a rank $n$ locally free sheaf and let $\xi\in 
{\rm{Ext}}^1(\mathcal{F},\mathcal{O}_X)$ be the class of 
the sequence (\ref{la sequenzasolita}). Let $\omega$ be an adjoint form 
associated to a subspace $W\subset   H^{0}(X,\det\sF)$. If $\omega\in\lambda^nW$ then 
$\xi\in\Ker({\rm{Ext}}^1(\mathcal{F},\mathcal{O}_X)\to 
{\rm{Ext}}^1(\mathcal{F}(-D_{W}),\mathcal{O}_X))$. Viceversa. Assume $\xi\in\Ker({\rm{Ext}}^1(\mathcal{F},\mathcal{O}_X)\to 
{\rm{Ext}}^1(\mathcal{F}(-D_{W}),\mathcal{O}_X))$. If 
$h^{0}(X,\sO_{X}(D_{W}))=1$, then 
$\omega\in\lambda^nW$. 
\end{thml}

See Theorem \ref{teoremaaggiunta} and Theorem \ref{viceversa}. 
The first part of the above theorem, but under the assumption $m=n$, is 
in \cite[Theorem 1.5.1]{PZ}.

By abuse of notation we denote by $|\det\sF|$ the linear system 
$\mP(H^{0}(X,\det\sF))$ and we denote by $D_{\det\sF}$ the 
fixed part of $|\det\sF|$. We set 
$|\det\sF|=D_{\det\sF}+| M_{\det\sF}|$. We prove a 
relation between liftability of adjoint forms and quadrics vanishing 
on the image of the map $\phi_{| M_{\det\sF}|}\colon X\dashrightarrow 
\mP(H^{0}(X,\det\sF)^{\vee})$ associated to 
$|\det\sF|$. We recall that for a
top form $\omega\in H^{0}(X,\det\sF)$ of $\sF$ to be liftable 
to an $n$-form of $\sE$ is equivalent to have that $\omega$ 
belongs to the kernel of the connecting homomorphism
\begin{equation*}
\partial_{\xi}^{n}\colon H^{0}(X, \det\sF)\to 
H^{1}(X,\bigwedge^{n-1}\sF)
\end{equation*} of the short exact sequence obtained by 
the wedge of the sequence (\ref{la sequenzasolita}); see: subsection \ref{sezione1}.

Denote by  $\lambda^{n}H^0(X,\sF)$ the image of 
$\bigwedge^{n}H^{0}(X,\sF)$ inside $H^{0}(X,\det\sF)$ and consider the linear system
$\mP(\lambda^{n}H^0(X,\sF))$; denote by $D_\sF$ its fixed 
component. Note that $D_{\rm{det}\sF}<D_{\sF}$, but in general 
$D_{\rm{det}\sF}\neq D_{\sF}$. We say that $\xi$ is a {\it{deformation supported on 
$D_{\sF}$}} if it 
belongs to the kernel of $\text{Ext}^1(\mathcal{F},\mathcal{O}_X)\to 
\text{Ext}^1(\mathcal{F}(-D_{\sF}),\mathcal{O}_X)$. 

\begin{defn}

    An \emph{adjoint quadric} for $\omega$ is a quadric in $\mP(H^0(X,\det\sF)^{\vee})$ of the form
    $$\omega^2=\sum L_i\cdot\omega_i,$$
    where $\omega$ is a $\xi$-adjoint of $W\subset  H^0(X,\sF)$, the 
    $\omega_i$'s are as above and $L_{i}\in H^0(X,\det\sF)$, 
    $i=1,\ldots, n+1.$
\end{defn}

 We point out that adjoint quadrics have rank less then or equal to 
 $2n+3$. We prove:

\begin{thml}    Let $X$ be an $m$-dimensional smooth variety. 
    Let $\sF$ be a %generically globally generated
		locally free sheaf of rank 
    $n$ such that $h^{0}(X, \sF)\geq n+1$ and let
    $\xi\in {\rm{Ext}}^1(\mathcal{F},\mathcal{O}_X)$. Assume that
    $\phi_{|M_{\det\sF} |}\colon X\dashrightarrow 
    \mP(H^{0}(X,\det\sF)^{\vee})$ is a non trivial rational map.
    If $\xi$ satisfies $\partial_\xi=0$ and $\partial_{\xi}^{n}=0$, 
    then $\xi$ is supported on $D_{\sF}$, provided that 
    there are no $\omega$-adjoint quadrics containing the schematic image $\phi_{|M_{\det\sF} 
    |}(X)$, where $\omega$ is a generic adjoint form. In particular 
    the claim holds if there are no quadrics of 
    rank less then or equal to $2n+3$ through the image.
\end{thml}
    
    See Corollary \ref{torelloo}. Note that if $D_{\sF}=0$ then 
    Theorem B gives a criterion for the vanishing of $\xi$. Theorem 
    B applies 
    directly if $Y:=\phi_{|M_{\det\sF} |}(X)$ is 
    a hypersurface of $\mP(H^{0}(X,\det\sF)^{\vee})$ of degree 
    $>2$; see Corollary \ref{suifibratiuno} for a more general claim.

Theorem A and Theorem B give a general criterion for a 
family $\pi\colon \sX\to B$ to have birational fibers, if the condition on 
the nature of adjoint quadrics given in 
Theorem B is satisfied; see: \ref{teoremalocaletorelli}. 
In this paper a {\it{family of relative dimension $n$}} 
is a smooth proper surjective morphism of smooth varieties 
$\pi\colon\sX\to B$ where any fiber $X_{y}:=\pi^{-1}(y)$ is a manifold of 
complex dimension $n$ and $B$ is a connected analytic variety. We say that $\pi\colon \sX {\rightarrow}B$ satisfies
    {\it{extremal liftability conditions over $B$}} if 
\begin{itemize}
\item[(i)] $ H^0({\sX},\Omega_{{\sX}}^1)
\twoheadrightarrow H^0(X_{b},\Omega_{X_{b}}^1)$;
\item[(ii)] $H^0({\sX},\Omega_{{\sX}}^n)
\twoheadrightarrow H^0(X_{b},\Omega_{X_{b}}^n)$
\end{itemize} where the notation \lq\lq$\twoheadrightarrow$\rq\rq means that the maps are surjective.
\noindent 
%Clearly if 
%$H^0(X_{b},\Omega_{X_{b}}^1)=H^0(X_{b},\Omega_{X_{b}}^n)=0$ the above 
%conditions are trivially satisfied. Instead
The families $\pi\colon\sX\to B$ which satisfy extremal liftability 
conditions and such that 
$H^0(X_{b},\Omega_{X_{b}}^1)\neq 0$, that is with irregular fibers, 
have the advantage that all the fibers have the same Albanese variety.
Now by a result called {\it {the Volumetric 
theorem}}, see: \cite[Theorem 1.5.3]{PZ}, cf. Theorem 
\ref{volteoremavolume}, we obtain the main application of 
our general theory that we present here:

\begin{thml}
    Let $\pi\colon\sX\to B$ be a family of $n$-dimensional irregular 
    varieties which satisfies extremal liftability conditions. Assume 
    that for every fiber $X$ the Albanese map of $X$ has degree $1$ 
    and that there are no adjoint quadrics containing the canonical image of $X$. Then 
    the fibers of $\pi\colon\sX\rightarrow B$ are birational. In particular 
    the claim holds if there are no quadrics of 
    rank less then or equal to $2n+3$ passing through the canonical image 
    of $X$.
\end{thml}

See: Theorem \ref{teoremalocaletorelli} for a slightly more general 
statement.

\subsection{Torelli-type problems}
As a partial motivation to 
understand the above general theory and in particular Theorem C, we point 
out the reader that if $\sF$ is the cotangent sheaf $\Omega^{1}_{X}$ 
of $X$, it gives, in some cases,  an 
answer to the generic Torelli problem and to the infinitesimal 
Torelli problem and, more importantly, that {\it{there are 
counterexamples}} to both generic and infinitesimal Torelli theorem if 
the above condition on the quadrics through the canonical image is 
removed. The {\it{global Torelli problem}} asks whether two compact K\"{a}hler 
manifolds  of dimension $n$, $X_{1}$, $X_{2}$, with an isomorphism $\psi:H^{n}(X_{1}, 
\mathbb Z)\rightarrow H^{n}(X_{2}, \mathbb Z)$ preserving the cup 
product pairing and with the same Hodge decomposition are biholomorphic. 
As references for the Torelli problems we quote the following books:
\cite{Vo1}, \cite{Vo2}, \cite{CMP}, see also \cite{Re2}. Here we 
need to recall only very few basic notions. 
Let $\pi\colon\sX\to B$ be a family. Fix a point $0\in B$. We can 
    interpret $\pi\colon\sX\to B$ as a deformation of complex 
    structures on $X:=X_{0}$. Indeed by Ehresmann's theorem, after 
    possibly shrinking $B$,  we have an isomorphism between 
$H^{k}(X_{y},\mathbb C)$ and $V:= H^{k}(X,\mathbb C) $, for every
$y\in B$. Furthermore, $\dim F^{p}H^{k}(X_{y},
\mathbb C)={\rm{dim}}_{\mathbb C}F^{p}H^{k}(X,
\mathbb C)=  b^{p,k}$, where 
$F^{p}H^{k}(X_{y}, \mathbb C)=\bigoplus_{ r\geq p}H^{r,k-r}(X_y)$.
The {\it{period map}} 
$$\sP^{p, k}\colon B\rightarrow \mathbb G={\rm{Grass}}(b^{p,k}, V)$$
(cf.~ \cite{gr-s}) is the map which to $y\in B$ associates the subspace
$F^{p}H^{k}(X_{y},\mathbb C)$ of $V$. 
In \cite{grif1}, \cite{grif2}, P. Griffiths 
proved that $\sP^{p, k}$ is holomorphic and that  the image of the differential
$d\sP^{p, k}:T_{B,{0}}\rightarrow T_{G, F^{p}H^{k}(X,\mathbb C)}$ 
is actually contained in 
$${\rm{Hom}}(F^{p}H^{k}(X,\mathbb C), 
F^{p-1}H^{k}(X,\mathbb
C)/F^{p}H^{k}(X,\mathbb C)).$$
\noindent
Setting  $q=k-p$ and using the canonical
isomorphism
$$
F^{p}H^{k}(X,\mathbb C)/F^{p+1}H^{k}(X,\mathbb C)\simeq
H^{q}(X,\Omega^{p}_{X})
$$
he showed that $d\sP^{p, k}$ is
the composition of the Kodaira-Spencer map $ T_{B,0}\rightarrow
H^{1}(X,\Theta_{X})$ with the map given by the cup product:
$$
d\pi^{q}_{p}\colon H^{1}(X,\Theta_{X})\rightarrow  
\Hom (H^{q}(X, \Omega^{p}_{X}), H^{q+1}(X, \Omega^{p-1}_{X}))
$$
where $\Theta_{X}$ is the tangent sheaf of $X$. 
Following \cite {Do} we recall that the {\em local Torelli problem} 
asks under which hypotheses the 
period map of $\pi\colon\sX\to B$ is an immersion and the {\it{generic 
Torelli problem}} asks whether the period map is generically injective where
$B$ is a connected analytic variety; the 
{\em n-infinitesimal Torelli problem} asks if the differential of the period map
$\sP^{n,n}$
for the local Kuranishi family $\pi \colon  \sX \to 
B$ (\cite{K}, cf. \cite{K-M}) is injective. 
We say that the {\em $n$-infinitesimal 
Torelli theorem holds for $X$} if $d\pi^{0}_n$ 
is injective.

In the case of smooth curves the answer to the Torelli problems is 
well-known; see: \cite{to}, \cite{andr}, 
\cite{we}, \cite{OS}.

Concerning Torelli problems in higher 
dimensions here we can recall only that some positive answers were 
obtained:
\cite{green1}, \cite{green2}, \cite{green3}, \cite{flenner}, 
\cite{cox}, \cite{Re1}, \cite{Pe}
and, with a slightly different point of 
view, \cite{Do}. See also \cite{Re2}, \cite{Vo1}, \cite{Vo2}.

On the other hand, surfaces $X$ of general 
type with $H^0(X,\Omega_{X}^1)=H^0(X,\Omega_{X}^2)=0$ are known examples of failure for the 
injectivity of the period map; see: \cite{Ca1}). 

More recently, the
problem has been studied in \cite{BC} where the authors construct  families 
such that $d\pi^{0}_{2}$ is not injective and the canonical sheaf is quasi 
very ample, that is the canonical map is a birational morphism and a 
local embedding on the complement of a finite set. In \cite{GZ} it is 
shown that for any natural 
number $N\geq 5$ there exists a 
generically smooth irreducible $(N+9)$-dimensional component $\sM_{N}$ 
of a moduli space of algebraic surfaces
such that for a general element $[X]$ of $\sM_{N}$, the canonical 
sheaf is very ample and $d\pi^{0}_2$ has kernel of dimension at least $1$. Hence 
the problem about good hypotheses to obtain the infinitesimal Torelli 
theorem is still quite open. Note that the counterexamples studied in \cite{BC} and in \cite{GZ} 
are given by fibered surfaces.

On the other hand by \cite{OS}, see 
also: \cite{Cod}, it is easy to see that  the 
$n$-th symmetric product of a hyperelliptic curve gives a 
counterexample to the infinitesimal Torelli theorem for non fibered 
varieties with Albanese morphism of degree $1$ onto the image. Note 
that if $X$ is the $2$-nd symmetric product of a hyperelliptic curve of genus $3$, 
$D_{\Omega^{1}_{X}}$ exists, it is a 
rational $-2$ curve; see: \cite[Proposition 3.17, (ii)]{CCM} and, 
as our theory indicates, the
infinitesimal deformations which kill the holomorphic forms are supported on 
$D_{\Omega^{1}_{X}}$.

It is here important to notice that the $n$-th symmetric product of 
hyperelliptic curves have canonical image 
contained in certain quadrics of rank less then or equal to $2n+3$. 
Indeed Theorem B confirms that for the Torelli problem $2n+3$ is a critical threshold for the 
rank of quadrics containing the canonical image. Our theory gives 
some positive answers to the infinitesimal 
Torelli theorem. In fact we easily deduce by Theorem B that it 
holds if $\Omega^{1}_{X}$ is 
generated by global sections, if we also assume that $X$ has irregularity 
$\geq n+1$ and there are no 
quadrics of rank less then or equal to $2n+3$ containing the canonical image; 
see: Corollary \ref{infinites}. 
We point out the reader that  
\cite{Pe} and  \cite{LPW} contain the same claim but under 
quite different hypotheses.
    
By the above counterexamples the hypotheses of Theorem C concerning 
the rank of the quadrics through the canonical image appear to 
be crucial to show the generic Torelli theorem. In dimension 
$\geq 3$ non-uniqueness of the minimal model is another possible obstruction. 

Hence in Corollary \ref{finali} we obtain that the generic Torelli 
theorem holds for families which satisfy the hypotheses of Theorem 
C and such that the general fiber is an $n$-dimensional irregular 
    minimal variety with unique minimal model. Even in the case of 
    surfaces there are obvious counterexamples to the above claim if 
    we do not assume minimality. It is sufficient to 
    consider the family $\pi\colon\sX\to B$ where $B$ is a smooth 
    curve inside a smooth surface $S$ and $X_{b}$ is the surface 
    obtained by the blow-up of $S$ at the point $b\in B$. So the 
    hypotheses of our Corollary \ref{finali} appear to be quite
    optimal to have a proof of the generic Torelli theorem.

\subsection{Acknowledgment}  
This research is supported by MIUR funds, 
PRIN project {\it Geometria delle variet\`a algebriche} (2010), coordinator A. Verra.

The authors would like to thank Miguel \'Angel Barja 
for very useful conversations on this topic.

\section{The adjoint theory}

We recall and generalize some of the results of \cite{PZ}.

\subsection{The Adjoint Theorem}
\label{sezione1}
Let $X$ be a compact complex
smooth variety of dimension $m$ and let $\mathcal{F}$ be a locally free sheaf of rank $n$. 
Fix an extension class $\xi\in \text{Ext}^1(\mathcal{F},\mathcal{O}_X)$ associated to the exact sequence:
\begin{equation}\label{unouno}
0\to\mathcal{O}_X\stackrel{d\epsilon}{\rightarrow} \mathcal{E}\stackrel{\rho_1}{\rightarrow} \mathcal{F}\to 0.
\end{equation}
The Koszul resolution associated to the section $d\epsilon\in H^0(X,\sE)
$\begin{equation*}
0\to \mathcal{O}_{X_0}\stackrel{d\epsilon}{\rightarrow} \mathcal{E}\stackrel{\wedge d\epsilon}{\rightarrow}\bigwedge^2\mathcal{E}\stackrel{\wedge d\epsilon}{\rightarrow}\cdots\stackrel{\wedge d\epsilon}{\rightarrow}\det\mathcal{E}\stackrel{\wedge d\epsilon}{\rightarrow}0
\end{equation*} splits into $n+1$ short exact sequences,
\begin{equation}
\label{seqesatta}
0\to\mathcal{O}_X\stackrel{d\epsilon}{\rightarrow} \mathcal{E}\stackrel{\rho_1}{\rightarrow} \mathcal{F}\to 0, 
\end{equation}
\begin{equation}
0\to\mathcal{F}\stackrel{d\epsilon}{\rightarrow} \bigwedge^2\mathcal{E}\stackrel{\rho_2}{\rightarrow} \bigwedge^2\mathcal{F}\to 0, 
\end{equation}
\begin{equation*}
\cdots
\end{equation*}
\begin{equation}
0\to\bigwedge^{n-1}\mathcal{F}\stackrel{d\epsilon}{\rightarrow} \bigwedge^n\mathcal{E}\stackrel{\rho_n}{\rightarrow} \det\mathcal{F}\to 0, 
\end{equation}
\begin{equation}
\label{isom}
0\to\det\mathcal{F}\stackrel{d\epsilon}{\rightarrow} \det\mathcal{E}\to0\to 0,
\end{equation}
each corresponding to $\xi$ via the natural isomorphism
\begin{equation*}
\text{Ext}^1(\sF,\sO_X)\cong \text{Ext}^1(\bigwedge^i\sF,\bigwedge^{i-1}\sF)
\end{equation*} where $i=0,\dots,n$. If we still denote by $\xi$ the element corresponding to $\xi$ by the isomorphism $\text{Ext}^1(\sF,\sO_X)\cong \text{Ext}^1(\sO_X,\sF^\vee)\cong H^1(X,\sF^\vee)$, then the coboundary homomorphisms
\begin{equation*}
\partial_\xi^i\colon H^0(X,\bigwedge^i\mathcal{F})\to H^1(X,\bigwedge^{i-1}\mathcal{F})
\end{equation*} are computed by the cup product with $\xi$ followed by contraction.

Denote by $H^n_{d\epsilon}$ the isomorphism inverse of (\ref{isom}) and by $\Lambda^{n+1}$ the natural map
\begin{equation}
\label{lambdan+1}
\Lambda^{n+1}\colon \bigwedge^{n+1}H^0(X,\mathcal{E})\to H^0(X,\det\mathcal{E}).
\end{equation} The composition of these homomorphisms defines
\begin{equation}
\Lambda:=H^{n}_{d\epsilon}\circ \Lambda^{n+1}\colon \bigwedge^{n+1}H^0(X,\mathcal{E})\to H^0(X,\det\mathcal{F}).
\end{equation} Let $W\subset \Ker(\partial_\xi^1)\subset H^0(X,\mathcal{F})$ be a vector subspace of
dimension $n+1$ and let $\mathcal{B}:=\{\eta_1,\ldots,\eta_{n+1}\}$
be a basis of $W$. By definition we can take liftings 
$s_1,\dots,s_{n+1}\in H^0(X,\mathcal{E})$ such that $\rho_1(s_i)=\eta_i$. If we consider the natural map
\begin{equation*}
\lambda^n\colon \bigwedge^{n}H^0(X,\mathcal{F})\to H^0(X,\det\mathcal{F}),
\end{equation*} we can define the subspace $\lambda^n W\subset H^0(X,\det\mathcal{F})$ generated by
\begin{equation*}
\omega_i:=\lambda^n(\eta_1\wedge\ldots\wedge\widehat{\eta_i}\wedge\ldots\wedge\eta_{n+1})
\end{equation*} for $i=1,\dots,n+1$.

\begin{defn}
The section
\begin{equation*}
\omega_{\xi,W,\mathcal{B}}:=\Lambda(s_1\wedge\ldots\wedge s_{n+1})\in H^0(X,\det\mathcal{F}).
\end{equation*} 
is called \emph{an adjoint form} of $\xi,W,\mathcal{B}$.
\end{defn}
\begin{defn}
The class
\begin{equation*}
[\omega_{\xi,W,\mathcal{B}}]\in \frac{H^0(X,\det\mathcal{F})}{\lambda^nW}
\end{equation*} is called \emph{an adjoint image} of $W$ by $\xi$.
\end{defn}

\begin{rem}
The class $[\omega_{\xi,W,\mathcal{B}}]$ depends on $\xi, W$ and 
$\mathcal{B}$ only. The form $\omega_{\xi,W,\mathcal{B}}$ depends 
also on the choice of the liftings $s_1,\dots,s_{n+1}$.
\end{rem}

\begin{rem}
\label{aggiuntazero}
If we consider another basis 
$\mathcal{B}'=\{\eta'_1,\ldots,\eta'_{n+1}\}$ of $W$, then $[\omega_{\xi,W,\mathcal{B}}]=k[\omega_{\xi,W,\mathcal{B}'}]$
where $k$ is the determinant of the matrix of the 
change of basis. In particular $[\omega_{\xi,W,\mathcal{B}}]=0$ if and only if $[\omega_{\xi,W,\mathcal{B}'}]=0$.
\end{rem}

\begin{rem}
\label{sollevamenti}
If $[\omega_{\xi,W,\mathcal{B}}]=0$ then we can find liftings 
$s_i\in H^0(X,\mathcal{E})$, $i=1,\dots,n+1$, such that $\Lambda^{n+1}(s_1\wedge \ldots\wedge s_{n+1})=0$ in $H^0(X,\det\mathcal{E})$. In particular with this choice we have $\omega_{\xi,W,\mathcal{B}}=0$ in $H^0(X,\det\mathcal{F})$.
\end{rem}
\begin{proof}
By hypothesis there exist $a_i\in \mathbb{C}$ such that
\begin{equation}
\label{ipotesi}
\omega_{\xi,W,\mathcal{B}}=\sum^{n+1}_{i=1} a_i\cdot \lambda^n(\eta_1\wedge\ldots\wedge\widehat{\eta_i}\wedge\ldots\wedge\eta_{n+1}).
\end{equation}
We can define a new lifting for the element $\eta_i$:
\begin{equation*}
\tilde{s_i}:=s_i+(-1)^{n-i}a_i\cdot d\epsilon.
\end{equation*} Now it is a trivial computation to show that $\Lambda^{n+1}(\tilde{s_1}\wedge \ldots\wedge \tilde{s}_{n+1})=0$.
\end{proof}

\begin{defn}
If $\lambda^nW$ is nontrivial we denote by $|\lambda^n W|\subset \mP(H^0(X,\det\mathcal{F}))$ the induced sublinear system. We call $D_W$ the fixed divisor of this linear system and $Z_W$ the base locus of its moving part $|M_W|\subset\mP(H^0(X,\det\mathcal{F}(-D_W)))$.
\end{defn}

From the natural map
$
\mathcal{F}(-D_W)\to\mathcal{F}
$ we have a homomorphism in cohomology
\begin{equation*}
H^1(X,\mathcal{F}^\vee)\stackrel{\rho}{\rightarrow} H^1(X,\mathcal{F}^\vee(D_W));
\end{equation*} we call $\xi_{D_W}=\rho(\xi)$. By obvious identifications the natural map
\begin{equation*}
\text{Ext}^1(\det\sF,\bigwedge^{n-1}\sF)\to\text{Ext}^1(\det\sF(-D_W),\bigwedge^{n-1}\sF)
\end{equation*} gives an extension $\sE^{(n)}$ and a commutative diagram:

\begin{equation}
\label{diagramma2}
\xymatrix { 
&&0\ar[d]&0\ar[d]\\
0 \ar[r] & \bigwedge^{n-1}\mathcal{F} \ar[r]\ar@{=}[d] &\mathcal{E}^{(n)} \ar[r]^-{\alpha} \ar[d]^{\psi}& \det\mathcal{F}(-D_W) \ar[d]^{\cdot D_W}\ar[r]&0\\
0 \ar[r] &\bigwedge^{n-1}\mathcal{F} \ar[r] &\bigwedge^{n}\mathcal{E} \ar[r]^{\rho_n}\ar[d] & \det\mathcal{F} \ar[r]\ar[d]&0 \\
&&\det\mathcal{F}\otimes_{\mathcal{O}_X}\mathcal{O}_{D_W}\ar@{=}[r]\ar[d]&\det\mathcal{F}\otimes_{\mathcal{O}_X}\mathcal{O}_{D_W}\ar[d]\\
&&0&0.
}
\end{equation}

We prove a slightly more general version of \cite[Theorem 1.5.1]{PZ}. Also the proof is slightly different because it does not use the Grothendieck duality.
\begin{thm}[Adjoint Theorem]
\label{teoremaaggiunta}
Let $X$ be a compact $m$-dimensional complex smooth variety. 
Let $\mathcal{F}$ be a rank $n$ locally free sheaf on $X$ and $\xi\in 
H^1(X,\mathcal{F}^\vee)$ the extension class of the exact sequence
(\ref{unouno}). Let $W\subset\Ker(\partial_\xi^1)\subset 
H^0(X,\mathcal{F})$ and $[\omega]$ one of its adjoint images.
If $[\omega]=0$ then $\xi\in\Ker(H^1(X,\mathcal{F}^\vee)\to H^1(X,\mathcal{F}^\vee(D_W)))$.
\end{thm}

\begin{proof} By remark \ref{aggiuntazero}, the vanishing of $[\omega]$
does not depend on the choice of a particular basis of $W$. Let 
$\mathcal{B}=\{ \eta_1,\dots,\eta_{n+1}\}$ be a  basis of $W$.

By hypothesis, $\omega\in \lambda^n W$, hence by remark \ref{sollevamenti} 
we can choose liftings $s_i\in H^0(X,\sE)$ of $\eta_i$ with $\Lambda^{n+1}(s_1\wedge \ldots\wedge s_{n+1})=0$. Consider
\begin{equation*}
s_1\wedge\ldots\wedge\widehat{s_i}\wedge\ldots\wedge s_{n+1}\in \bigwedge^nH^0(X,\mathcal{E})
\end{equation*} and define $\Omega_i$ its image in $H^0(X,\bigwedge^n\mathcal{E})$ through the natural map.
By construction 
$
\rho_n(\Omega_i)=\omega_i,
$ where we remind the reader that
$
\omega_i:=\lambda^n(\eta_1\wedge\ldots\wedge\widehat{\eta_i}\wedge\ldots\wedge\eta_{n+1}).
$ Since $D_W$ is the fixed divisor of the linear system $|\lambda^n W|$ and the sections $\omega_i$ generate this linear system, then the $\omega_i$ are in the image of
\begin{equation*}
\det\mathcal{F}(-D_W)\stackrel{\cdot D_W}{\longrightarrow}\det\mathcal{F}, 
\end{equation*} and so we can find sections $\tilde{\omega}_i\in H^0(X,\det\mathcal{F}(-D_W))$ such that 
\begin{equation}
\label{d}
d\cdot\tilde{\omega}_i=\omega_i,
\end{equation} where $d$ is a global section of $\mathcal{O}_X(D_W)$ with $(d)=D_W$. Hence, using the commutativity of 
(\ref{diagramma2}), we can find liftings $\tilde{\Omega}_i\in 
H^0(X,\mathcal{E}^{(n)})$ of the sections $\Omega_i$, $i=1,\ldots,n+1$.

The evaluation map
\begin{equation*}
\bigoplus_{i=1}^{n+1}\mathcal{O}_X\stackrel{\tilde{\mu}}{\longrightarrow}\mathcal{E}^{(n)}
\end{equation*} given by the global sections $\tilde{\Omega}_i$, composed with the map $\alpha$ of 
(\ref{diagramma2}), induces a map $\mu$ which fits into the following diagram
\begin{equation*}
\xymatrix { &&\bigoplus_{i=1}^{n+1}\mathcal{O}_X\ar[d]^{\tilde{\mu}} \ar@{=}[r]&\bigoplus_{i=1}^{n+1}\mathcal{O}_X\ar[d]^{\mu}\\
0 \ar[r] & \bigwedge^{n-1}\mathcal{F} \ar[r] &\mathcal{E}^{(n)} \ar[r]^-\alpha & \det\mathcal{F}(-D_W) \ar[r]&0.
}
\end{equation*} The morphism $\mu$ is given by the multiplication by 
$\tilde{\omega}_i$ on the $i$-th component. 
Consider the sheaf $\Ima\tilde{\mu}$. Locally, on an open subset
$U$, we can write 
\begin{equation}\label{sezloc}
s_{i|U}=\sum_{j=1}^n b_j^i\cdot \sigma_j+c^i\cdot d\epsilon,
\end{equation} where $(\sigma_1,\dots,\sigma_n,d\epsilon)$ is a family of local generators of $\sE$. By $\Lambda^{n+1}(s_1\wedge \ldots\wedge s_{n+1})=0$ we have
\begin{equation*}
\begin{vmatrix}
	b^2_1&\dots&b^2_n\\
	\vdots &  \ddots & \vdots  \\
	b^{n+1}_1&\dots&b^{n+1}_n
\end{vmatrix}c^1+ \cdots+(-1)^{n}\begin{vmatrix}
	b^1_1&\dots&b^1_n\\
	\vdots &  \ddots & \vdots \\
	  b^{n}_1&\dots&b^n_n
\end{vmatrix}c^{n+1} =0.
\end{equation*} Obviously also
\begin{equation*}
\begin{vmatrix}
	b^2_1&\dots&b^2_n\\
	\vdots &  \ddots & \vdots  \\
	b^{n+1}_1&\dots&b^{n+1}_n
\end{vmatrix}b^1_j+ \cdots+(-1)^{n}\begin{vmatrix}
	b^1_1&\dots&b^1_n\\
	\vdots &  \ddots & \vdots \\
	  b^{n}_1&\dots&b^n_n
\end{vmatrix}b^{n+1}_j =0
\end{equation*} for $j=1,\dots,n$, where it is easy to see that
\begin{equation*}
\begin{vmatrix}
	b^2_1&\dots&b^2_n\\
	\vdots &  \ddots & \vdots  \\
	b^{n+1}_1&\dots&b^{n+1}_n
\end{vmatrix},\ldots, \begin{vmatrix}
	b^1_1&\dots&b^1_n\\
	\vdots &  \ddots & \vdots \\
	  b^{n}_1&\dots&b^n_n
\end{vmatrix}
\end{equation*} 
are respectively the local expressions of the sections 
$\omega_1,\dots,\omega_{n+1}$. Now let $d_U$ be a local equation of $d\in 
H^0(X,\mathcal{O}_X(D_W))$. By (\ref{d}) we have 
\begin{equation*}
\begin{vmatrix}
	b^2_1&\dots&b^2_n\\
	\vdots &  \ddots & \vdots  \\
	b^{n+1}_1&\dots&b^{n+1}_n
\end{vmatrix}=d_U\cdot f_1,
\end{equation*}
\begin{equation*}
\vdots
\end{equation*}
\begin{equation*}
\begin{vmatrix}
	b^1_1&\dots&b^1_n\\
	\vdots &  \ddots & \vdots \\
	  b^{n}_1&\dots&b^n_n
\end{vmatrix}=d_U\cdot f_{n+1}
\end{equation*} where the functions $f_i$ are local expressions of 
$\tilde{\omega_i}$. Then we have
\begin{equation*}
d_U\cdot (f_1\cdot c^1+ \cdots+(-1)^{n} f_{n+1}\cdot c^{n+1}) =0
\end{equation*} and
\begin{equation*}
d_U\cdot (f_1\cdot b^1_j+ \cdots+(-1)^{n} f_{n+1}\cdot b^{n+1}_j) =0
\end{equation*} for $j=1,\dots,n$. Since by definition $d_U$ vanishes on $D_W\cap U$, then on $U$ 
\begin{equation*}
f_1\cdot c^1+ \cdots+(-1)^{n} f_{n+1}\cdot c^{n+1} =0
\end{equation*} and
\begin{equation*}
f_1\cdot b^1_j+ \cdots+(-1)^{n} f_{n+1}\cdot b^{n+1}_j =0
\end{equation*}for $j=1,\dots,n$. We immediately obtain
\begin{equation}
f_1\cdot s_1+\dots+(-1)^{n}f_{n+1}\cdot s_{n+1}=0.
\end{equation} By definition  the scheme $Z_W\cap U$ is given by $(f_1=0,\dots,f_{n+1}=0)$. Let $P\in U$ be a point not in $\text{supp}(Z_W)$. At least one of the functions $f_i$ can be inverted in a neighbourhood of $P$, for example let the germ of $f_1$ be nonzero in $P$. We have then a relation
\begin{equation*}
s_1=g_2\cdot s_2+\dots+g_{n+1}\cdot s_{n+1}
\end{equation*} and so we can easily find holomorphic functions $h_i$ such that
\begin{equation*}
\Omega_i=h_i\cdot \Omega_1
\end{equation*} for $i=2,\dots,n+1$. Since
\begin{equation*}
\mathcal{E}^{(n)}\stackrel{\psi}{\rightarrow}\bigwedge^n\mathcal{E}
\end{equation*} is injective, then we have
\begin{equation*}
\tilde{\Omega}_i=h_i\cdot \tilde{\Omega}_1
\end{equation*} for $i=2,\dots,n+1$. The section $\tilde{\Omega}_1$ is nonzero, otherwise $\Omega_i=0$ for $i=1,\dots,n+1$ and then also $\omega_i=0$ for $i=1,\dots,n+1$, but this fact contradicts our hypothesis that $\lambda^nW$ is not trivial. In particular we have proved that the sheaf $\Ima\tilde{\mu}$ has rank one outside $Z_W$. Furthermore, since it is a
subsheaf of the locally free sheaf $\mathcal{E}^{(n)}$, then $\Ima\tilde{\mu}$ is torsion free. Denote $\Ima\tilde{\mu}$ by $\sL$.

By definition
\begin{equation*}
\Ima\mu:=\det\mathcal{F}(-D_W)\otimes \mathcal{I}_{Z_W}.
\end{equation*}

The morphism 
\begin{equation*}
\alpha\colon\mathcal{E}^{(n)}\to\det\mathcal{F}(-D_W)
\end{equation*} restricts to a surjective morphism, that we continue to call $\alpha$,
\begin{equation*}
\sL\stackrel{\alpha}{\rightarrow}\Ima\mu,
\end{equation*} between two sheaves that are locally free of rank one 
outside $Z_W$. The kernel of $\alpha$ is then a torsion subsheaf of 
$\sL$, which is torsion free, hence $\alpha$ gives an isomorphism 
$\sL\cong\Ima\mu$.

Since $X$ is normal (actually it is smooth) and $Z_W$ has codimension at least 2, we have that the inclusion
\begin{equation*}
\mathcal{E}^{(n)}\supset\mathcal{L}^{\vee\vee}\cong\det\mathcal{F}(-D_W)
\end{equation*} gives the splitting
\begin{equation*}
\xymatrix { 0\ar[r] &\bigwedge^{n-1}\mathcal{F}\ar[r]&\mathcal{E}^{(n)}\ar[r]&
\det\mathcal{F}(-D_W)\ar[r]\ar@/_1.5pc/[l]&0.
}
\end{equation*}
Since $\xi_{D_W}$ is the element of $H^1(X,\mathcal{F}^\vee(D_W))$ associated to this extension, we conclude that $\xi_{D_W}=0$.

\end{proof}

As an easy consequence of the Adjoint Theorem we have the 
infinitesimal Torelli theorem for primitive varieties with 
$p_{g}=q=n+1$; for the notion of primitiveness see \cite[Definition 
1.2.4]{victor1}, see also our introduction to section \ref{sezionefamiglie} of this 
paper.

\begin{cor}\label{casainperiferia} Let $X$ be an $n$-dimensional primitive variety of general type such that 
    $p_{g}=q=n+1$. If $\Omega_{X}^{1}$ is generated by global sections then
    the $1$-infinitesimal Torelli theorem holds for $X$.
 \end{cor}
 \begin{proof} Take $W=H^0(X,\Omega^1_X)$. Since $X$ is primitive,
     $\lambda^nW=\left\langle \omega_1,\dots,\omega_{n+1}\right\rangle=H^0(X,\omega_X).$
 By the hypothesis $W=\Ker\partial_{\xi}$, we can construct an adjoint form $\omega_{\xi,W,\sB}$ and obviously $\omega_{\xi,W,\sB}\in\lambda^nW$. Hence the claim holds by Theorem \ref{teoremaaggiunta}.
\end{proof}
The cases studied in Corollary \ref{casainperiferia} do occur.

\begin{cor} Let $X$ be a surface of general type such that 
    $p_{g}=q=3$ and $K_{X}^{2}=6$. 
    Then the $1$-infinitesimal Torelli theorem holds for $X$.
\end{cor} 
\begin{proof}
Indeed a surface of general type such that 
    $p_{g}=q=3$ and $K_{X}^{2}=6$ is the symmetric product of a 
    non-hyperelliptic curve of genus $3$ by \cite[Proposition 3.17, (i)]{CCM}. 
    Hence our proof of Corollary \ref{casainperiferia} applies.
\end{proof}
\begin{rem} Note that for surfaces of general type which are the 
symmetric product of a hyperelliptic curve of genus $3$, the infinitesimal Torelli 
theorem does not hold by \cite{OS}. Indeed $D_{\Omega^{1}_{X}}$ exists and it is a 
rational $-2$ curve; see: \cite[Proposition 3.17, (ii)]{CCM}. Hence an 
infinitesimal Torelli deformation is supported on 
$D_{\Omega^{1}_{X}}$.
\end{rem}

\subsection{An inverse of the Adjoint Theorem}

Let $X$ be a smooth compact complex variety of dimension $m$ and $\mathcal{F}$
a locally free sheaf of rank $n$. Let $\xi\in 
H^1(X,\mathcal{F}^\vee)$ be the class associated to the extension 
(\ref{unouno}). Fix $W=\left\langle \eta_1,\dots,\eta_{n+1}\right\rangle\subset\Ker(\partial_\xi^1)\subset H^0(X,\mathcal{F})$
a subspace of dimension $n+1$ and take an adjoint form $\omega=\omega_{\xi,W,\mathcal{B}}$ for chosen
$s_1,\dots,s_{n+1}\in H^0(X,\mathcal{E})$ liftings of respectively 
$\eta_1,\dots,\eta_{n+1}$. As above set $|\lambda^n W|=D_W+|M_W|$. 
Note that by construction $\omega\in H^{0}(X, 
\det\sF\otimes\sO_{X}(-D_{W}))$.

\begin{thm}\label{viceversa} Assume that 
    $h^0(X,\mathcal{O}_X(D_W))=1$. If $\xi_{D_W}=0$ then $[\omega]=0$.
\end{thm}
\begin{proof}
If $n=m=1$ then \cite{CP}[Theorem 1.1.8.] shows that $\xi_{D_W}=0$ iff $[\omega]=0$.

Let $(\omega)_0=D_W+C$ be the decomposition of the adjoint divisor in 
its fixed component and in its moving one. The first step is to construct a global section 
\begin{equation*}
\Omega'\in H^0(X,\bigwedge^n\mathcal{E}(-C))
\end{equation*} which restricts, through the natural map
$
\bigwedge^n\mathcal{E}(-C)\to\det\mathcal{F}(-C),
$ to $d\in H^0(X,\det\mathcal{F}(-C))$, where $(d)_0=D_W$. Indeed consider the commutative diagram:
\begin{equation*}
\xymatrix { 
&0\ar[d]&0\ar[d]&0\ar[d]\\
0\ar[r]&\bigwedge^{n-1}\mathcal{F}(-C)\ar[r]\ar[d]&\bigwedge^{n}\mathcal{E}(-C)\ar[r]^-{G_2}\ar[d]^-{G_1}&\det\mathcal{F}(-C)\ar[r]\ar[d]&0\\
0\ar[r]&\bigwedge^{n-1}\mathcal{F}\ar[r]^-{d\epsilon}\ar[d]^-{H_1}&\bigwedge^{n}\mathcal{E}\ar[r]^-{\rho_n}\ar[d]&\det\mathcal{F}\ar[r]\ar[d]&0\\
0\ar[r]&\bigwedge^{n-1}\mathcal{F}\otimes_{\mathcal{O}_X}\mathcal{O}_C\ar[r]^-{H_2}\ar[d]&\bigwedge^{n}\mathcal{E}\otimes_{\mathcal{O}_X}\mathcal{O}_C\ar[r]^-{H_3}\ar[d]&\det\mathcal{F}\otimes_{\mathcal{O}_X}\mathcal{O}_C\ar[r]\ar[d]&0\\
&0&0&0
}
\end{equation*}
The adjoint form $\omega\in H^0(X,\det\mathcal{F})$ vanishes when 
restricted to $C$. Since $\xi_{D_{W}}=0$ there exists a lifting 
$\Omega\in H^0(X,\bigwedge^{n}\mathcal{E})$ of $\omega$. Indeed, by (\ref{diagramma2}), take a global lifting $c\in H^0(X, \det\sF(-D_W))$ of $\omega$. Since $\xi_{D_W}=0$, $c$ can be lifted to a section $e\in H^0(X, \sE^{(n)})$. Define $\Omega:=\psi(e)$. By commutativity,
$
H_3(\Omega|_C)=0.
$ Hence there exists $\bar{\mu}\in H^0(X,\bigwedge^{n-1}\mathcal{F}\otimes_{\mathcal{O}_X}\mathcal{O}_C)$ such that
$
H_2(\bar{\mu})=\Omega|_C
$. To construct $\Omega'$ we first show that $\delta(\bar{\mu})=\xi_{D_{W}}$ where
\begin{equation*}
\delta\colon 
H^0(X,\bigwedge^{n-1}\mathcal{F}\otimes_{\mathcal{O}_X}\mathcal{O}_C)\to H^1(X,\bigwedge^{n-1}\mathcal{F}(-C)).
\end{equation*} 
To do that we consider the cocycle $\{\xi_{\alpha\beta}\}$ which gives 
$\xi$ with respect to an open cover $\{U_\alpha\}$ of $X$. For any 
$\alpha$ there 
exists  $\gamma_\alpha\in 
\bigwedge^{n-1}\mathcal{F}(U_\alpha)$ such that:
\begin{equation}
\label{Omega}
\Omega|_{U_\alpha}=\omega_\alpha+\gamma_\alpha\wedge d\epsilon
\end{equation} where $\omega_\alpha$ is obtained by the analogous 
expression to the one in equation (\ref{sezloc}) but on 
$\bigwedge^{n}\sE$. A local computation on $U_\alpha\cap U_\beta$ shows that
$
\omega_\beta-\omega_\alpha=\xi_{\alpha\beta}(\omega)\wedge d\epsilon,
$ where the notation $\xi_{\alpha\beta}(\omega)$ indicates the 
natural contraction. Since we have a global lifting of $\omega$,
this cocycle must be a coboundary. Hence by (\ref{Omega})
\begin{equation*}
(\gamma_\beta-\gamma_\alpha)\wedge d\epsilon=\omega_\beta-\omega_\alpha=\xi_{\alpha\beta}(\omega)\wedge d\epsilon,
\end{equation*} that is
$
\gamma_\beta-\gamma_\alpha=\xi_{\alpha\beta}(\omega)
$. Note that the sections $\gamma_\alpha$ are local liftings of 
$\bar{\mu}$. Indeed by the injectivity of $H_2$, it is enough to show that
$
(\gamma_\alpha\wedge d\epsilon)|_C=\Omega|_C,
$ and this is obvious by (\ref{Omega}) and by the fact that 
$
\omega_\alpha|_C=0
$. Since the morphism
\begin{equation*}
\bigwedge^{n-1}\mathcal{F}(-C)(U_\alpha\cap U_\beta)\to \bigwedge^{n-1}\mathcal{F}(U_\alpha\cap U_\beta)
\end{equation*} is given by multiplication by $c_{\alpha\beta}$, where $c_{\alpha\beta}$ is a local equation of $C$
on $U_\alpha\cap U_\beta$, then the desired expression of the coboundary $\delta(\bar{\mu})$ is given by the cocycle
\begin{equation*}
\frac{\gamma_\beta-\gamma_\alpha}{c_{\alpha\beta}}=\frac{\xi_{\alpha\beta}(\omega)}{c_{\alpha\beta}}.
\end{equation*}
By the above local computation, we see that this cocycle gives 
$\xi_{D_W}$ via the isomorphism
\begin{equation}
\label{isom3}
H^1(X,\bigwedge^{n-1}\mathcal{F}(-C))\cong H^1(X,\mathcal{F}^\vee(D_W)).
\end{equation}
This is easy to prove since $\xi_{D_W}$ is locally given by
$
\xi_{\alpha\beta}\cdot d_{\alpha\beta},
$ where $d_{\alpha\beta}$ is the equation of $D_W$ on $U_{\alpha}\cap 
U_{\beta}$, so the image of $\xi_{D_W}$ through the isomorphism 
(\ref{isom3}) is
$
\frac{\xi_{\alpha\beta}\cdot d_{\alpha\beta}\cdot c_{\alpha\beta}}{c_{\alpha\beta}}=\frac{\xi_{\alpha\beta}(\omega)}{c_{\alpha\beta}},
$ which is exactly $\delta(\bar{\mu})$.
Now we use again our hypothesis $\xi_{D_W}=0$ to write 
$\delta(\bar{\mu})=0$. Then there exists a global section $\mu \in 
H^0(X,\bigwedge^{n-1}\mathcal{F})$ which is a lifting of $\bar{\mu}$. 
This gives the global section
\begin{equation*}
\tilde{\Omega}:=\Omega-\mu\wedge d\epsilon\in H^0(X,\bigwedge^{n}\mathcal{E}).
\end{equation*} By construction $\tilde{\Omega}$ is a new lifting of 
$\omega$ which now vanishes once restricted to C:
\begin{equation*}
\tilde{\Omega}|_C=\Omega|_C-\mu\wedge d\epsilon|_C=\Omega|_C-H_2(\bar{\mu})=0.
\end{equation*} The wanted section
$
\Omega'\in H^0(X,\bigwedge^{n}\mathcal{E}(-C))
$ is the global section which lifts $\tilde{\Omega}$ and by 
construction satisfies
$
\rho_n(G_1(\Omega'))=\omega
$ and
$
G_2(\Omega')=d.
$

The second step is to use $\Omega'$ to show that $[\omega]=0$. The global sections
$
\omega_i:=\lambda^n(\eta_1\wedge\ldots\wedge\widehat{\eta_i}\wedge\ldots\wedge\eta_{n+1})\in H^0(X,\det\mathcal{F})
$ generate $\lambda^nW$ and by definition they vanish on $D_W$, that is there exist global sections $\tilde{\omega}_i\in H^0(X,\det\mathcal{F}(-D_W))$ such that
$
\omega_i=\tilde{\omega}_i\cdot d
$. We consider now the commutative diagram
\begin{equation}
\label{diagramma3}
\xymatrix { 0\ar[r] &\mathcal{O}_X(-C)\ar[r]^-\alpha\ar[d]^-{\cdot C}& W\otimes\mathcal{O}_X\ar[r]^-\gamma\ar[d]^-\beta & \bar{\mathcal{F}}\ar[r]\ar[d]^-\iota&0\\
0\ar[r]&\mathcal{O}_X\ar[r]^-{d\epsilon}&\mathcal{E}\ar[r]^{\rho_1}&\mathcal{F}\ar[r]&0.
}
\end{equation} The homomorphism $\beta$ is locally defined by
\begin{equation*}
(f_1,\dots,f_{n+1})\mapsto (-1)^nf_1\cdot s_1+\dots+f_{n+1}\cdot s_{n+1}
\end{equation*} and similarly $\rho_1\circ\beta$ is given by
\begin{equation*}
(f_1,\dots,f_{n+1})\mapsto (-1)^nf_1\cdot \eta_1+\dots+f_{n+1}\cdot \eta_{n+1}.
\end{equation*}
The homomorphism $\iota\circ\gamma$ is the factorization of this map and 
it defines $\bar{\mathcal{F}}$. The homomorphism $\alpha$ is defined in the following way: if $f\in\mathcal{O}_X(-C)(U)$
is a local section, then $\tilde{\omega}_i$ are sections of the dual sheaf $\mathcal{O}_X(C)$ and, locally on $U$, $\alpha$ is given by
\begin{equation*}
f\mapsto (\tilde{\omega}_1(f),\ldots,\tilde{\omega}_{n+1}(f)).
\end{equation*}
We want to verify that the first square is commutative. If we write 
locally $\eta_i=\sum_{j=1}^n b_j^i\cdot \sigma_j$, and
\begin{equation*}
s_i=\sum_{j=1}^n b_j^i\cdot \sigma_j+c^i\cdot d\epsilon,
\end{equation*} then 
\begin{align*}
&\beta(\alpha(f))=\beta\left(f_U\begin{vmatrix}
	b^2_1&\dots&b^2_n\\
	\vdots &  \ddots & \vdots  \\
	b^{n+1}_1&\dots&b^{n+1}_n
\end{vmatrix}\frac{1}{d_U},\ldots,f_U\begin{vmatrix}
	b^1_1&\dots&b^1_n\\
	\vdots &  \ddots & \vdots \\
	  b^{n}_1&\dots&b^n_n
\end{vmatrix}\frac{1}{d_U} \right)=\\&=(-1)^nf_U\begin{vmatrix}
	b^2_1&\dots&b^2_n\\
	\vdots &  \ddots & \vdots  \\
	b^{n+1}_1&\dots&b^{n+1}_n
\end{vmatrix}\frac{1}{d_U}s_1+\cdots +f_U\begin{vmatrix}
	b^1_1&\dots&b^1_n\\
	\vdots &  \ddots & \vdots \\
	  b^{n}_1&\dots&b^n_n
\end{vmatrix}\frac{1}{d_U}s_{n+1}=\\
&=f_U\begin{vmatrix}
	b^1_1&\dots&b^1_n&c^1\\
	b^2_1&\dots&b^2_n&c^2\\
	\vdots &  \ddots & \vdots  & \vdots \\
	b^{n+1}_1&\dots&b^{n+1}_n&c^{n+1}
\end{vmatrix}\frac{1}{d_U}d\epsilon=\\&=(f_U\cdot c_U)d\epsilon,
\end{align*} where $f_U$ and $d_U$ are local holomorphic functions 
which represent 
the sections $f$ and $d$ respectively. 
The first equality uses the fact that the determinants appearing
in the first line are the local equations of the sections $\omega_i$
(see also the proof of Theorem \ref{teoremaaggiunta}); 
the last equality comes from the fact that the determinant in the 
second to last line is the local equation of $\omega$, and 
$\omega=d\cdot c$. To dualize diagram (\ref{diagramma3}) we recall 
the sheaves isomorphisms $\mathcal{F}^\vee\cong \bigwedge^{n-1}\mathcal{F}(-C-D_W)$ and 
$\mathcal{E}^\vee\cong \bigwedge^{n}\mathcal{E}(-C-D_W)$. Moreover we 
also recall the isomorphism $W^\vee\cong\bigwedge^{n}W$,  
given by
\begin{equation*}
\eta^i\mapsto \eta_1\wedge\ldots\wedge\widehat{\eta_i}\wedge\ldots\wedge\eta_{n+1}
\end{equation*} where $\eta^1,\ldots,\eta^{n+1}$ is the basis of $W^\vee$ dual to the basis $\eta_1,\ldots,\eta_{n+1}$ of $W$. Define
\begin{equation*}
e_i:=\eta_1\wedge\ldots\wedge\widehat{\eta_i}\wedge\ldots\wedge\eta_{n+1}.
\end{equation*}
Now we dualize (\ref{diagramma3}):
\begin{equation}\label{analogo}
\xymatrix { 0\ar[r]&\bar{\mathcal{F}}^\vee \ar[r]^-{\gamma^\vee}& \bigwedge^{n}W\otimes\mathcal{O}_X\ar[r]^-{\alpha^\vee} & \det\mathcal{F}(-D_W)\\
0\ar[r]&\bigwedge^{n-1}\mathcal{F}(-C-D_W)\ar[r]\ar[u]&\bigwedge^{n}\mathcal{E}(-C-D_W)\ar[r]\ar[u]^{\beta^\vee}&\mathcal{O}_X\ar[u]^-{\cdot C}\ar[r]&0.
}
\end{equation}
Here $\alpha^\vee$ is the evaluation map given by the global sections 
$\tilde{\omega}_i$, not necessarily surjective. Nevertheless we 
tensor by $\mathcal{O}_X(D_W)$ and we obtain
\begin{equation*}
\xymatrix {  \bigwedge^{n}W\otimes H^0(X,\mathcal{O}_X(D_W))\ar[r]^-{\overline{\alpha^\vee}} & H^0(X,\det\mathcal{F})\\
H^0(X,\bigwedge^{n}\mathcal{E}(-C))\ar[r]\ar[u]^{\overline{\beta^\vee}}&H^0(X,\mathcal{O}_X(D_W)).\ar[u]^-{\cdot C}
}
\end{equation*}
The section $\Omega'\in H^0(X,\bigwedge^{n}\mathcal{E}(-C))$ 
constructed in the first part of the proof gives in 
$H^0(X,\det\mathcal{F})$ the adjoint $\omega$. By commutativity
\begin{equation*}
\omega=\overline{\alpha^\vee}(\overline{\beta^\vee}(\Omega')).
\end{equation*} By our hypothesis $h^0(X,\mathcal{O}_X(D_W))=1$, the section $d$ is a basis of $H^0(X,\mathcal{O}_X(D_W))$, so $(e_1\otimes d,\ldots,e_{n+1}\otimes d)$ is a basis of $\bigwedge^{n}W\otimes H^0(X,\mathcal{O}_X(D_W))$. We have then
\begin{equation*}
\overline{\beta^\vee}(\Omega')=\sum_{i=1}^{n+1}c_i\cdot e_i\otimes d
\end{equation*} where $c_i\in \mathbb{C}$ and
\begin{equation*}
\omega=\overline{\alpha^\vee}(\overline{\beta^\vee}(\Omega'))=\overline{\alpha^\vee}(\sum_{i=1}^{n+1}c_i\cdot e_i\otimes d)=\sum_{i=1}^{n+1}c_i\cdot\tilde{\omega}_i\cdot d=\sum_{i=1}^{n+1} c_i\cdot\omega_i,
\end{equation*} and hence $[\omega]=0$.
\end{proof}

\section{An infinitesimal Torelli-type theorem}

Theorem \ref{teoremaaggiunta} can be used to show splitting criteria 
for extension classes of %{\it{generically globally generated}}
locally free sheaves.
To this aim we consider  a locally free sheaf $\sF$ of rank $n$ 
over an $m$-dimensional smooth variety $X$. %We remind the
%reader that generically globally generated means that $\sF$ is 
%generated by the global sections outside a loci of codimension at 
%least $1$.
Naturally associated to $\sF$ there is the invertible sheaf 
$\det\sF$ and  the natural homomorphism:
\begin{equation}\label{immaginelmbdan}
\lambda^n\colon \bigwedge^n  H^0(X,\sF)\to  H^0(X,\det\sF).
\end{equation}
We denote by  $\lambda^{n}H^0(X,\sF)$ its image. Consider the linear system
$\mP(\lambda^{n}H^0(X,\sF))$ and recall that $D_\sF$ is its fixed 
component and $|M_{\sF}|$ is its associated mobile linear system. %By \cite[Proposition 3.1.6]{PZ}, if $W$
%is a generic $n+1$-dimensional subspace of $H^0(X,\sF)$, then $D_{\sF}=D_W$.
Moreover we denote by $|\det\sF|$ the linear system 
associated to $\det\sF$ and by $D_{\det\sF}$, $M_{\det\sF}$ respectively 
its fixed and its movable part; that is: 
$|\det\sF|=D_{\det\sF}+|M_{\det\sF}|$. Finally note 
that $D_{\det\sF}$ is a sub-divisor of $D_\sF$.

\begin{defn}

    An \emph{adjoint quadric} for $\omega$ is a quadric in $\mP(H^0(X,\det\sF)^{\vee})$ of the form
    $$\omega^2=\sum L_i\cdot\omega_i,$$
    where $\omega$ is a $\xi$-adjoint of $W\subset  H^0(X,\sF)$, 
    $\omega_i:=\lambda^n(\eta_1\wedge\ldots\wedge\widehat{\eta_i}\wedge\ldots\wedge\eta_{n+1})\in H^0(X,\det\sF)$ and $L_i\in H^0(X,\det\sF)$.
\end{defn}

Obviously an adjoint quadric has rank less then or equal to $2n+3$, and, if it
exists, it is 
constructed by the extension class $\xi$.

\begin{thm}\label{torelloo1}
    Let $X$ be an $m$-dimensional smooth variety. 
    Let $\sF$ be a %generically globally generated
		locally free sheaf of rank 
    $n$ such that $h^{0}(X, \sF)\geq n+1$, let $\xi\in 
    H^1(X,\sF^\vee)$ and let $Y$ be the schematic 
    image of $\phi_{|M_{\det\sF} |}\colon X\dashrightarrow 
    \mP(H^{0}(X,\det\sF)^{\vee})$. 
    If $\xi$ is such that 
    $\partial_{\xi}^{n}(\omega)=0$, where $\omega$ is an adjoint form associated
    to an $n+1$-dimensional subspace $W\subset\Ker\partial_\xi\subset 
    H^0(X,\sF)$, then $[\omega]=0$, providing that there are no 
    $\omega$-adjoint quadrics.
\end{thm}
\begin{proof}
Let $\sB=\{\eta_{1},\ldots, \eta_{n+1}\}$ be 
a basis of $W$. Set 
$\omega_i$ for $i=1,\dots,n+1$ as above and denote by $\tilde{\omega}_i\in 
H^0(\det\sF(-D_W)\otimes \sI_{Z_W})$ the corresponding sections 
via $0\to H^0(X, \det\sF(-D_W)\otimes \sI_{Z_W})\to H^0(X,\det\sF)$.
Recall that $\lambda^nW:=\left\langle 
\omega_1,\dots,\omega_{n+1}\right\rangle\subset 
H^0(X,\det\mathcal{F})$ is the vector space  generated by the sections $\omega_i$. The standard evaluation
     map $\bigwedge^{n}W\otimes\sO_{X}\to 
     \det\sF(-D_W)\otimes \sI_{Z_{W}}$ given by 
     $\tilde{\omega}_1,\dots,\tilde{\omega}_{n+1}$ results in the following exact sequence
     
\begin{equation}\label{genera1}
\xymatrix { 0\ar[r]&\sK \ar[r]&
\bigwedge^{n}W\otimes\mathcal{O}_X\ar[r] & \det\mathcal{F}(-D_W)\otimes \sI_{Z_W}\ar[r]&0}
\end{equation} which is associated to the class $\xi'\in \text{Ext}^1(\det\sF(-D_W)\otimes \sI_{Z_W},\sK)$.
This sequence fits into the following commutative diagram (cf. (\ref{analogo}))
\begin{equation}
\xymatrix { 0\ar[r]&\sK \ar[r]& \bigwedge^{n}W\otimes\mathcal{O}_X\ar[r] & \det\mathcal{F}(-D_W)\otimes \sI_{Z_W}\ar[r]&0\\
0\ar[r]&\sF^\vee\ar[r]\ar[u]&\sE^\vee\ar[r]\ar[u]^{f}&\mathcal{O}_X\ar[u]^{g}\ar[r]&0,
}
\end{equation} where $f$ is the map given by the contraction by the sections $(-1)^{n+1-i}s_i$, for $i=1,\dots,n+1$, and $g$ is given by the global section $\sigma \in H^0(X,\det\sF(-D_W)\otimes \sI_{Z_W})$ corresponding to the adjoint $\omega$.
We have the standard factorization
\begin{equation}
\xymatrix { 0\ar[r]&\sK \ar[r]& \bigwedge^{n}W\otimes\mathcal{O}_X\ar[r] & \det\mathcal{F}(-D_W)\otimes \sI_{Z_W}\ar[r]&0\\
0\ar[r]&\sK \ar[r]\ar@{=}[u]&\sL\ar[r]\ar[u]&\sO_X\ar[u]^{g}\ar@{=}[d]\ar[r]&0\\
0\ar[r]&\sF^\vee\ar[r]\ar[u]&\sE^\vee\ar[r]\ar[u]&\mathcal{O}_X\ar[r]&0
}
\end{equation} where the sequence in the middle is associated to the class
$\xi''\in H^1(X,\sK)$ which is the image of $\xi\in H^1(X,\sF^\vee)$ through the map $H^1(X,\sF^\vee)\to H^1(X,\sK)$.
In particular we obtain the commutative square
\begin{equation}
\xymatrix {H^0(X,\det \sF(-D_W)\otimes \sI_{Z_W})\ar[r]&H^1(X,\sK)\ar@{=}[d]\\
H^0(X,\sO_X)\ar[u]\ar[r]& H^1(X,\sK).
}
\end{equation} By commutativity we immediately have that the image of $\sigma$ through
the coboundary map $H^0(X,\det \sF(-D_W)\otimes \sI_{Z_W})\to H^1(X,\sK)$ is $\xi''$.
Tensoring by $\det \sF$, the map $\sF^\vee\to\sK$ gives 
\begin{equation}
\xymatrix { \sF^\vee\otimes\det\sF \ar@{=}[d]\ar[r] &\sK\otimes\det\sF\\
\bigwedge^{n-1}\sF\ar[ur]^{\Gamma}
}
\end{equation} and, since $\xi\cdot\omega\in 
H^1(X,\sF^\vee\otimes\det\sF)$ is sent to $\xi''\cdot\omega\in H^1(X,\sK\otimes\det\sF)$, we have that
\begin{equation}
H^{1}(\Gamma)(\xi\cup \omega)=\xi''\cdot\omega,
\end{equation} where $\xi\cup \omega$ is the cup product.

By hypothesis $\partial_\xi^n(\omega)=\xi\cup \omega=0\in H^1(X,\bigwedge^{n-1}\sF)$, so also 
$\xi''\cdot \omega=0\in H^1(X,\sK\otimes\det\sF)$, hence the global section $\sigma\cdot\omega\in H^0(X,\det \sF(-D_W)\otimes \sI_{Z_W}\otimes\det\sF)$ is in the kernel of the coboundary map $H^{0}(X,\det \sF(-D_W)\otimes \sI_{Z_W}\otimes\det\sF)\to H^{1}(X, 
\sK\otimes\det\sF)$ associated to the sequence

\begin{equation}
\xymatrix { 0\ar[r]&\sK\otimes\det\sF \ar[r]&
\bigwedge^{n}W\otimes\det\sF\ar[r] & \det\mathcal{F}(-D_W)\otimes \sI_{Z_W}\otimes\det\sF\ar[r]&0.}
\end{equation}
This occurs iff there exist 
$L^{\sigma}_{i}\in 
H^{0}(X,\det\sF)$, $i=1,\ldots, n+1$ such that
\begin{equation}
\sigma\cdot\omega=\sum_{i=1}^{n+1}L^{\sigma}_{i}\cdot\tilde{\omega}_{i}.   
\end{equation}
This relation gives the following relation in $H^{0}(X,\det
\sF^{\otimes2})$:
\begin{equation}\label{quadra}
\omega\cdot\omega=\sum_{i=1}^{n+1}L^{\sigma}_{i}\cdot\omega_{i}.
\end{equation}
Assume now that the adjoint form $\omega$ is not in the vector space 
$\lambda^nW$. Then the equation (\ref{quadra}) gives an adjoint 
quadric. By contradiction the claim follows.
\end{proof}

\begin{cor}
\label{torelloooo} Under the above hypotheses, $\xi$ is a supported 
deformation on $D_{W}$; that is, $\xi_{D_{W}}$ is trivial.
\end{cor}
\begin{proof}
The claim follows by Theorem 
\ref{torelloo1} and by Theorem \ref{teoremaaggiunta}. 
\end{proof}

The following is Theorem B of the introduction:
\begin{cor}
\label{torelloo}
Under the hypotheses of Theorem \ref{torelloo1}, if we further assume that $W$ is generic, it follows that 
$\xi$ is a deformation supported on $D_{\sF}$, that is 
$\xi_{D_{\sF}}$ is trivial.
\end{cor}
\begin{proof}
By \cite[Proposition 3.1.6]{PZ}, if $W$
is a generic $n+1$-dimensional subspace of $H^0(X,\sF)$, we have that $D_{\sF}=D_W$.
Then the claim follows by Corollary \ref{torelloooo}.
\end{proof}

Note that there are cases where our hypothesis applies easily:

\begin{cor}\label{suifibratiuno}
Let $X$ be an $m$-dimensional smooth variety. 
Let $\sF$ be a %generically globally generated 
locally free sheaf of rank 
    $n$. Assume that $\phi_{|M_{\det\sF} |}\colon X\dashrightarrow
    \mP(H^{0}(X,\det\sF)^{\vee})$ is a non trivial rational map such 
    that its schematic image
    is a complete intersection of 
    hypersurfaces of degree 
    $>2$. Let $\xi\in H^1(X,\sF^\vee)$. If 
$\partial_\xi =0$ and $\partial_{\xi}^{n}(\omega)=0$ 
where $\omega$ is an adjoint form associated to a generic $n+1$-dimensional subspace $W\subset H^0(X,\sF)$, then $\xi$ is a 
deformation supported on $D_{\sF}$. In particular if $D_{\sF}=0$ then $\xi=0$.
\end{cor}
\begin{proof}
The claim follows directly by Corollary \ref{torelloo}.
\end{proof}

\begin{cor}\label{infinites} Let $X$ be an $n$-dimensional variety of general type with 
    irregularity $\geq n+1$ and such that its cotangent 
    sheaf is generated by global sections. Suppose also that there are no adjoint 
    quadrics containing the canonical image of $X$. Then the infinitesimal Torelli 
    theorem holds for $X$. In particular it holds if there are no quadrics of 
    rank less then or equal to $2n+3$ passing through the canonical image 
    of $X$.
\end{cor}

\begin{proof}
By Corollary \ref{torelloo}, any $\xi\in H^{1}(X,\Theta_{X})$ such that 
$\partial_\xi =0$ and $\partial_{\xi}^{n}(\omega)=0$, 
where $\omega$ is an adjoint form associated to a generic 
$n+1$-dimensional subspace $W\subset H^0(X,\Omega^{1}_{X})$, is supported on the 
branch locus of the Albanese morphism and, since we have assumed it to 
be trivial, then the trivial infinitesimal deformation is the only 
possible case. 
\end{proof}

\begin{rmk}\label{guardaecco} As a typical 
    application we obtain infinitesimal Torelli for a smooth divisor 
    $X$ of an $n+1$-dimensional Abelian variety $A$ such that 
    $p_{g}(X)=n+2$; for explicit examples consider the case of a 
    smooth surface $X$
    in polarization $(1,1,2)$ in an abelian $3$-fold $A$. The 
    invariants of $X$ are $p_{g}(X)=4$, $q(X)=3$ and $K^{2}=12$. The 
    canonical map is in general a birational morphism onto a surface 
    of degree $12$. See \cite[Theorem 6.4]{CS}.
\end{rmk}

%\begin{rmk} Note that if the canonical map of $X$ is an embedding then our 
%   condition on the rank of the quadrics containing the variety 
%    $X=Y$ is up to now not so clear; but see the end of the 
%    introduction of this paper.
%  \end{rmk}

\section{Families with birational fibers}
\label{sezionefamiglie}

Theorem \ref{torelloo1} provides a criterion to understand which
families of irregular varieties of general type have birational 
fibers. First we recall some basic definitions.

\subsection{Families of morphisms}

 A 
 {\it{family}} of $n$-dimensional varieties is a flat smooth 
proper morphism $\pi\colon\sX\rightarrow B$ such that the {\it{fiber}} $X_b :=\pi^{-1}(b)$ over a point  $b\in B$ has dimension $n$. The 
variety $B$ is called {\it{base of the family}}. We only assume that $B$ is smooth connected and
analytic.  We restrict our study to the case where $X_{b}$ is an 
irregular variety of
general type. We recall that any irregular variety $X$ comes equipped 
with its Albanese variety ${\rm{Alb}} (X)$ and its Albanese morphism 
${\rm{alb}}(X):X_{}\rightarrow {\rm{Alb}}(X_{})$. 

A {\it{fibration}} is a surjective proper flat morphism $f \colon X\to Z$
 with connected fibers between the smooth varieties $X$ and $Z$. A fibration  $f \colon X\to Z$ is {\it{irregular}} if $Z$ is an irregular variety.

We recall that a smooth  irregular variety $X$ is said to be of {\it{maximal Albanese dimension}}
if ${\rm{dim}}\, {\rm{alb}}(X) = {\rm{dim}}\, X$. If ${\rm{dim}}\, {\rm{alb}}(X) = {\rm{dim}}\, X$ and
 ${\rm{alb}}\colon X\to {\rm{Alb}}(X)$ is {\it not surjective}, that is $q(X)> {\rm{dim}}\, X$, $X$ is said to be of {\it{Albanese general type}}. 
A fibration $f \colon X \to Z$  is called a {\it{higher  irrational pencil}} if $Z$ is of Albanese general type. 
An irregular variety $X$ is said to be{\it{ primitive}} if it does not admit any higher irrational pencil; see: \cite[Definition 
1.2.4]{victor1}. 
Note that irregular fibrations (resp. higher irrational pencils) are higher-dimensional analogous 
to fibrations over non-rational curves (resp. curves of genus $g \geq 2$).

To study a family $\pi\colon\sX\rightarrow B$ of
irregular varieties it 
is natural to consider the case where it comes equipped with a family 
$p\colon \sA\rightarrow B$ of Abelian varieties; that is, the fiber $A_{b}:=p^{-1}(b)$ is an Abelian variety
of dimension $a>0$.
%**************DEFINIZIONE CHIAVE***********************
\begin{defn}\label{chiave} Let $\pi\colon\sX\rightarrow B$ be a family of irregular varieties
    of general type and $p:\sA\rightarrow B$ a family of Abelian varieties. A morphism
$\Phi\colon \sX{\rightarrow}\sA$ will be called a
{\it family of Albanese type} over $B$ if:
\begin{itemize}
\item[(i)] $\Phi$ fits into the
commutative diagram
$$\begin{array}{rcl}
&\sX\stackrel{\Phi}
{\longrightarrow}\sA&\\
\pi \!\!\!\!\!\!\!\!& \searrow \,\,\,\swarrow &\!\!\!\!\!\!\!\!\! p\\
& \!B. &
\end{array} $$

\item[(ii)] The induced map
$\phi_{b}\colon X_{b}\rightarrow A_{b}$ of $\Phi$ on
$X_{b}$ is birational onto its image $Y_{b}.$

\item[(iii)] The cycle $Y_{b}$ generates the fiber $A_{b}$ as a group.

\end{itemize}\noindent
\end{defn}
%************************DEFINIZIONE CHIAVE: FINE*********************
\noindent See: \cite[Definition 1.1.1.]{PZ}. We remark that $a>n$ ( cf. see
\cite[ p.311 and Corollary to Theorem 10.12,(i)]{Ii}).
We shall say that two families over $B$,
$\Phi: \sX \rightarrow \sA $ and
$\Psi:\sY \rightarrow \sA $, of Albanese type
have the same  {\it{image}} if it is true
fiberwise, that is $\phi_b(X_b)=\psi_b(Y_b)$ for every $b\in B$.
\medskip
\noindent

\noindent
{\it{Base change.}} Albanese type families have a good behaviour
under base
change. In fact let
$\mu:B^{'}\rightarrow B$ be a base change, then
$\mu^{\ast}(\Phi)=\Phi\times
{\rm{id}}:\sX\times_{B} B^{'}\stackrel{}{\rightarrow}\sA\times_{B} B^{'}$
is an
Albanese type family over $B^{'}.$ In particular for a connected
subvariety
$C\hookrightarrow B$ the base change of $\Phi\colon \sX{\rightarrow}\sA$ to $C$ is well defined and
it will be denoted by $\Phi_{C}\colon \sX_{C}\to\sA_{C}.$
\medskip

%*************** SEIONI E EQUIVALENZA****************
\noindent{\it{Translated family.}} If $s\colon B\to\sA$ is a section of
$p:\sA\rightarrow B$, we define the translated family
$\Phi_{s}:\sX\rightarrow
\sA$ of $\Phi$ by the formula:
$$\Phi_{s}(x)=\Phi(x)+ s(\pi(x)).$$
Notice that $\Phi_{s}:\sX\rightarrow
\sA$ is a family of Albanese type.
Two
families
$\Phi$ and $\Psi$ over $B$
are said to be {\it{translation equivalent}}
if there exists a section $\sigma$ of $p$ such that
the {\it{images}} of
$\Phi_{\sigma}$
and $\Psi$ (fiberwise) coincide.
%*************** SEIONI E EQQUIVALENZA: FINE***********************

We are interested in a condition that guarantees that, up
to restriction, the fibers of the restricted 
family are birationally equivalent. For
the reader's convenience we recall the following definition given in 
\cite[definition 1.1.2]{PZ}:

%***************DEFINIZIONE DI EQUIVALENZA*******************
\begin{defn}\label{panna}  Two families of Albanese type
$\Phi\colon \sX{\rightarrow}\sA$,
$\Phi'\colon \sX^{'} {\rightarrow}\sA^{'}$ over, respectively, $B$ and $B'$ will be said
{\it{locally equivalent}}, if there exist an open set $U\subset B$ an open set $U'\subset B'$ and a biregular map $\mu \colon U'\to U:=\mu(U')\subset B$ such that the pull-back families 
$\mu^{\ast}( \Phi_{U})$ and $\Phi^{'}_{U^{'}}$ are {\rm{translation equivalent}}. We will say that $\Phi$ is {\it{trivial}} if $\sX=X\times B$,
$\sA=A\times B$ and $\pi_{A}(\Phi({X_{b}}))=\pi_{A}(\Phi(X_{b_{0}}))$ for all $b$ where $\pi_{A}\colon A\times B\to A$ is the natural projection.
\end{defn}
%***************DEFINIZIONE DI EQUIVALENZA: FINE*******************

%**************POLARIZZAZIONE*************
\noindent
{\it{Polarization.}} 
We recall that a polarized variety is a couple $(X,H)$ where $X$ is a 
variety and $H$ is a big and nef divisor on it. The Albanese type family $\Phi$
provides the fibers $A_{b}$ with a natural polarization. Letting
$\omega$ a Chern form of the canonical divisor $K_{X_{b}}$,
the Hermitian form of the polarization $\Xi_{b}$ is defined
on $H^{1,0}(A_{b})$ by:
\begin{equation}\label{polarizzazione}
\Xi_{b}(\eta_{1},\eta_{2})=\int_{X_{b}}
\phi_{b}^{\ast}\eta_{1}\wedge\phi_{b}^{\ast}{\overline{\eta_{2}}}\wedge\omega^
{n-1}.
\end{equation}\noindent Notice that, in this way,
$p:\sA\rightarrow B$ becomes a family of {\it{polarized}}
Abelian varieties. There is then a suitable moduli variety ${\bf{\sA_{a}}}$
parameterizing polarized
Abelian variety of dimension $a$, and a holomorphic map, called the 
period map,
$P\colon B\to {\bf{\sA_{a}}}$ defined by $ P(b)= (A_b, \Xi_b)$. 
% and the differential
% \begin{equation}\label{periodi}
% d\sP:T_{B,b}\rightarrow T_{{\bf{\sA_{a}}},(A_b, \Xi_b)}.
% \end{equation}
% \noindent
Finally we remark that if $\sP(B)$ is a point, then, up to shrinking $B,$
$p:\sA\rightarrow B$ is equivalent to the trivial family, $\sA\rightarrow A\times B$ where $A$ is the fixed Abelian variety
corresponding to $\sP(B)$.  \medskip

%**************POLARIZZAZIONE: FINE*************

%**********************ESEMPIO DELLA FIBRAZIONE DI ALBANESE***********

\noindent{\it{Example.}} The standard example of Albanese type family is given by a family 
$\pi:\sX\rightarrow B$ with a section $s\colon B\rightarrow \sX$. 
Indeed by  $s\colon B\rightarrow \sX$  we have a 
family $p\colon {\rm{Alb}}(\sX)\rightarrow B$ whose fiber is
$p^{-1}(b)={\rm{Alb}}(X_{b})$; the section gives
a family $\Phi\colon \sX\rightarrow{\rm{Alb}}(\sX)$
with fiber:
$${\rm{alb}}(X_{b}):X_{b}\rightarrow {\rm{Alb}}(X_{b}).$$
\noindent
If we also assume that $\phi_{b}={\rm{alb}}(X_{b})$ has
degree $1$ onto the image then $\Phi\colon 
\sX\rightarrow{\rm{Alb}}(\sX)$ is an Albanese type family. We will call
$\sX\stackrel{\Phi}{\rightarrow}{\rm{Alb}}(\sX)$ an {\it{Albanese
family}}.
\medskip
%****************FINE ESEMPIO DELLA FIBRAZIONE DI ALBANESE************
\medskip

We will use the following:
\begin{prop}\label{pomodoro}
An Albanese type family $\Phi\colon\sX\to\sA$ is locally equivalent to a trivial family if
and only if the fibers $X_{b}$ are birationally equivalent.
\end{prop}
\begin{proof} See \cite[Proposition 1.1.3]{PZ}.\end{proof}

\subsection{Families with liftability conditions}
To find conditions which force the fibers of a family 
$\pi\colon\sX\rightarrow B$ to be birationally equivalent, it is 
natural to understand conditions on Albanese type families 
whose associated family of Abelian varieties is trivial. The easiest 
condition to think of is given by the condition of liftability for any $1$-form.

\begin{prop}\label{insalata}
Let $\Phi\colon\sX\to\sA$ be an Albanese type family such that for 
every $b\in B$ the map $ H^0({\sX},\Omega_{{\sX}}^1)
\to H^0(X_{b},\Omega_{X_{b}}^1)$ is surjective. Then up to shrinking 
$B$ the fibers of 
$p\colon\sA\to B$ are isomorphic.
\end{prop}
\begin{proof} Let $\mu_b\in {\rm{Ext}}^1(\Omega^{1}_{A_b},\sO_{A_b})$
    be the class given by the family $p\colon\sA\to B$,
 that is the class of the following extension:
$$
0\to \sO_{A_{b}}\to\Omega^1_{\sA|A_{b}}\to\Omega^1_{A_{b}}\to 0.
$$\noindent
Since $\phi_b$ is flat we have the following sequence
$$
0\to \phi_b^*\sO_{A_{b}}\to\phi_b^*\Omega^1_{\sA|A_{b}}\to\phi_b^*\Omega^1_{A_{b}}\to 0
$$ which fits into the following diagram
\begin{equation*}
\xymatrix { 0\ar[r]&\phi_b^*\sO_{A_{b}} \ar[r]\ar@{=}[d]&
\phi_b^*\Omega^1_{\sA|A_{b}}\ar[r]\ar[d] & \phi_b^*\Omega^1_{A_{b}}\ar[d]\ar[r]&0\\
0\ar[r]&\sO_{X_b} \ar[r]&
\Omega^1_{\sX|X_{b}}\ar[r] & \Omega^1_{X_{b}}\ar[r]&0.
}
\end{equation*}
In cohomology we have
\begin{equation*}
\xymatrix {H^0(X_b,\phi_b^*\Omega^1_{A_{b}})\ar[d]\ar[r]&H^1(X_b, \sO_{X_b})\ar@{=}[d]\\
H^0(X_b,\Omega^1_{X_{b}})\ar[r]&H^1(X_b, \sO_{X_b})
}
\end{equation*} so, by commutativity and by the hypothesis $ H^0({\sX},\Omega_{{\sX}}^1)
\twoheadrightarrow H^0(X_{b},\Omega_{X_{b}}^1)$, we immediately obtain $ H^0(X_b,\phi_b^*\Omega^1_{\sA|A_{b}})
\twoheadrightarrow H^0(X_{b},\phi_b^*\Omega^1_{A_{b}})$ and hence the coboundary $\partial_{\mu_{b}}\colon H^0(A_b ,\Omega^1_{A_{b}})\to H^1(A_b, \sO_{A_{b}})$ is trivial.
%By hypothesis $\partial_{\mu_{b}}\colon H^0(A_b ,\Omega^1_{A_{b}})\to H^1(A_b, \sO_{A_{b}})$
%is trivial since, by
%the universal property of the Albanese morphism, and by condition 
 %$(iii)$ of definition \ref{chiave} we can construct the projection 
 %$H^0(X_{b},\Omega_{X_{b}}^1)=H^{0}({\rm{Alb}}(X_{b}), 
 %\Omega^{1}_{ {\rm{Alb}} (X_{b}) }  )\twoheadrightarrow H^0(\Omega^1_{A_{b}})$ and 

%\begin{equation*}
%\xymatrix { &&     H^0(\sX, \Omega^{1}_{\sX} )       \ar[d] \ar[r]&  H^0(\sA, \Omega^{1}_{\sA}  )  \ar[d]\\
%& &   H^0(X_{b},\Omega_{X_{b}}^1)  \ar[r]&     H^0(A_{b},\Omega_{A_{b}}^1)        &}
%\end{equation*} 
%\noindent
%is commutative. 
Then by cf. \cite [Page  78]{CP} we conclude.
\end{proof}

For an Albanese type family
$\Phi\colon\sX\to\sA$ such that $ H^0({\sX},\Omega_{{\sX}}^1)
\twoheadrightarrow H^0(X_{b},\Omega_{X_{b}}^1)$ we can say more. 
Actually up to shrinking $B$ it is trivial to show that for every 
$b\in B$ it holds that ${\rm{Alb}}(X_b)=A$ where $A$ is a {\it{fixed}} Abelian 
variety. For later reference we state:

\begin{cor}\label{musica} Let $\Phi\colon\sX\to\sA$ be an Albanese type family such that  $ H^0({\sX},\Omega_{{\sX}}^1)
\twoheadrightarrow H^0(X_{b},\Omega_{X_{b}}^1)$ where $b\in B$. Then for every $b\in B$ there exists an open neighbourhood $U$ 
such that the restricted family $\Phi_{U}\colon\sX_{U}\to\sA_{U}$ is locally equivalent to 
$\Psi\colon \sX_{U}\to{ \widehat{A}}\times U$, where $\widehat{A}$ is an Abelian variety. Moreover there exists another Abelian variety $A$ 
such that for every $b\in U$ it holds that $A={\rm{Alb}}(X_b)$ and the natural morphism $ A\to\widehat A$ gives a factorization of
$\Psi \colon\sX_U\to \widehat A\times U$ via 
${\rm{alb}}(\sX_{U})\colon\sX_{U}\to A\times U$.
\end{cor}

We recall a definition given in the introduction of this paper and suitable to study Torelli type problems for irregular varieties.

%***************TORELLI FAMIGLIE******************

\begin{defn}\label{torellifamily} We say that a family $f\colon \sX {\rightarrow}B$ of relative dimension $n$ 
    satisfies extremal liftability conditions over $B$ if
\begin{itemize}
\item[(i)] $ H^0({\sX},\Omega_{{\sX}}^1)
\twoheadrightarrow H^0(X_{b},\Omega_{X_{b}}^1)$;
\item[(ii)] $H^0({\sX},\Omega_{{\sX}}^n)
\twoheadrightarrow H^0(X_{b},\Omega_{X_{b}}^n)$. 
\end{itemize}
\noindent 
\noindent 
\end{defn}
This definition means that all the $1$-forms and all the $n$-forms of the fibers are obtained by restriction from the family $\sX$. Comparing the two conditions with the hypotheses of Theorem \ref{torelloo1}, they ensure that $\partial_\xi=0$ and $\partial_\xi^n=0$.

%************************DEFINIZIONE CHIAVE: FINE*********************

\begin{rem} Let $X$ be a smooth variety such that ${\rm{alb}}(X)\colon X\to{\rm{Alb}}(X)$ has degree $1$. Let 
$f\colon X \longrightarrow Z$ be a fibration of relative dimension $n$ such that  the general fiber $f^{-1}(y)=X_{y}$ is smooth
of general type and with irregularity $q\geq n+1$. 
By rigidity of Abelian subvarieties, the image of the map $Alb(X_y) \longrightarrow Alb(X)$ is a (translate of a) fixed abelian variety $A$. If, moreover, we
have surjections  $H^0(X,\Omega_X^1) \twoheadrightarrow
H^0(X_y,\Omega_{X_y}^1)$ and  $H^0(X,\Omega_X^n) \twoheadrightarrow
H^0(X_,\Omega_{X_y}^n)$, then $A=Alb(X_y)$ and taking a
sufficiently small polydisk around any smooth fibre of $f$ we obtain a
family which satisfies extremal liftability conditions.
\end{rem}

%****************TORELLI FAMIGLIE FINE**************

\subsection{The theorem}

Our main theorem is based on the Volumetric Theorem; see \cite[Theorem 1.5.3]{PZ}:

\begin{thm}\label{volteoremavolume}
    Let $\Phi\colon\sX\to\sA$ be an Albanese type family such that 
    $p\colon\sA\rightarrow B$ has fibers isomorphic to a fixed Abelian variety $A$.
    Let $W\subset 
    H^{0}(A,\Omega^{1}_{A})$ be a generic $n+1$-dimensional subspace 
    and $W_{b}\subset H^{0}(X_{b},\Omega^{1}_{X_{b}})$ its pull-back 
    over the fiber $X_{b}$. Assume that for every point $b\in B$ it holds that $\omega_{_{\xi_{b},W_{b},
\sB_{b}}}\in\lambda^{n}W_{b}$ where $\xi_{b}\in 
H^{1}(X_{b},\Theta_{X_{b}})$ is the class given on $X_{b}$ by 
$\pi\colon\sX\to B$, then the fibers of $\pi\colon\sX\to B$ are birational.
\end{thm}
\proof See \cite[Theorem 1.5.3]{PZ}.\qed

\medskip
The following is our main result. It is a slightly more general 
version of Theorem C stated in the introduction. Note that it applies also to the case where the canonical 
linear system of a general fiber has fixed divisorial components and non trivial base locus.

\begin{thm}\label{teoremalocaletorelli} 
    Let $\Phi\colon\sX\to\sA$ be a family of Albanese type whose 
    associated family of $n$-dimensional irregular varieties $\pi\colon\sX\rightarrow 
    B$ satisfies extremal liftability conditions.
 Assume that there are no adjoint quadrics through the canonical 
 image of any fiber $X$ of $\pi\colon\sX\rightarrow B$. Then 
    the fibers of $\pi\colon\sX\rightarrow B$ are birational.
\end{thm}
\begin{proof} Since our claim is local in the analytic category, up to base change, we can assume
    that $B$ is a $1$-dimensional disk and that $\pi\colon\sX\rightarrow B$ has 
    a section. More precisely we take two points in $B$ and a curve $C$ connecting them, then make a base change from $B$ to $C$. 
    Therefore in the rest of this proof $B$ is a curve. By proposition \ref{insalata} 
    $p\colon\sA\to B$ is trivial. Moreover the Albanese family ${\rm{alb}}(\sX)\colon \sX\to  {\rm{Alb}}(\sX)$ 
    exists and by proposition \ref{pomodoro}, our claim is equivalent to show that  the Albanese family 
    ${\rm{alb}} (\sX) \colon \sX\to  {\rm{Alb}} (\sX)$
 is locally equivalent to the trivial family. By Corollary 
 \ref{musica}  we can assume that ${\rm{Alb}}(\sX)=A\times B$ too.

We denote by $\xi_b\in H^1(X_b,\Theta_{X_b})$  the class associated to the infinitesimal deformation of $X_b$ induced
by $\pi\colon\sX\rightarrow B$. First we assume that $q>n+1$.

Let $\sB:=\{dz_1,\ldots,dz_{n+1}\}$ be a 
basis of an $n+1$-dimensional generic subspace $W$ of $H^0( A,\Omega^{1}_{A})$. 
For every $b\in B$ let $\eta_i(b):= {\rm{alb}}(X_b)^{\star}dz_i$, $i=1,\ldots, n+1$. By standard
theory of the Albanese morphism it holds that $\sB_b:=\{ 
\eta_1(b),\ldots, \eta_{n+1}(b)\}$ is a basis of the pull-back $W_{b}$ 
of $W$ inside $H^0(X_b,\Omega_{X_{b}}^1)$. 
Let $\omega_i (b):= 
\lambda^{n}(\eta_1(b)\wedge\ldots\wedge\eta_{i-1}(b)\wedge{\widehat{\eta_i(b)}}\wedge\ldots\wedge\eta_{n+1}(b))$ for $i=1,\ldots, n+1$. 
Since $\Phi\colon\sX\to\sA$ is a family of Albanese type, 
${\rm{dim}}\lambda^nW_{b}\geq 1$, actually if $q>n+1$ by \cite[Theorem 1.3.3]{PZ} it follows
that $\lambda^nW_{b}$ has dimension $n+1$,
and we can write:
$\lambda^nW_{b}=\langle \omega_1(b),\ldots,\omega_{n+1}(b)\rangle$. 
Let $\omega_b:=\omega_{\xi_{b},W_{b}, \sB_{b}}$ be an adjoint image of $W_b$.

By Theorem \ref{torelloo1} it follows that $\omega_b\in\lambda^nW_{b}$. By Theorem \ref{volteoremavolume} we 
conclude. 
\end{proof}

The above theorem gives, in particular, an answer to the generic Torelli problem if 
the fibers of $\pi\colon\sX\rightarrow B$ are smooth minimal with unique 
minimal model. Indeed we have:

\begin{cor}\label{finali}
    Let $\pi\colon\sX\to B$ be a family of $n$-dimensional irregular 
    varieties which satisfies extremal liftability conditions. Assume 
    that every fiber $X$ is minimal, it has a unique minimal 
    model and its Albanese morphism has degree $1$. If there are no 
    adjoint quadrics containing the canonical image of $X$, then 
    the generic Torelli theorem holds for $\pi\colon 
    \sX\to B$ (assuming that the Kodaira-Spencer map is generically 
    injective). In particular the claim holds if no quadric of rank $\leq 2n+3$ 
    contains the canonical image of $X$.
\end{cor}

\begin{proof} Since the fiber $X$ has a unique minimal model the 
    claim follows directly by Theorem \ref{teoremalocaletorelli}.
    \end{proof}

\subsection{Examples}

As a very simple exercise on our theory, the reader can prove 
the infinitesimal Torelli theorem for the 
product $X$ of two non-hyperelliptic curves of genus $3$ or a general 
curve of genus $4$ using Corollary \ref{torelloooo} and without assuming 
its validity in the case of curves.

Theorem C, and consequently Corollary \ref{finali},  applies directly to the
families whose fiber $X$ has canonical map {\it{which is not}} an 
isomorphism and $X$ is an irregular variety such that its canonical image $Y$
is a (possibly very singular) hypersurface of degree $>2$ 
or a (possibly very singular) complete intersection
of hypersurfaces of degree $>2$: see also 
our remark \ref{guardaecco}. 
Many of these examples are not well studied in the literature.

%As far as other possible examples are concerned, 
%we remind the reader that the space $\sQ_{k,n}$ of quadrics of the 
%projective space $\mP^{n}$ of rank $\leq k$ has dimension 
%$k(n-k+1)+\binom{k+1}{2}-1$. If we consider, 
%for example, minimal irregular surfaces whose canonical map is (to greatly simplify) an embedding, the 
%canonical image is $2$-normal, and for the Neron Severi group it holds ${\rm{NS}}(S)=[K_S]\mathbb{Z}$, then, letting $n=p_{g}(S)-1$, by 
%Riemann-Roch theorem the dimension of the vector space of quadrics containing 
%$S=Y$ is
%$h^{0}(\mP^{n},\sI_{S/\mP^{n}}(2))=
%\binom{n+2}{2}-K^{2}_{S}-\chi(\sO_{S})$. In the surface case $k=2n+3=7$. 
%Now, imposing the natural condition $K^{2}_{S}+\chi(\sO_{S})\geq 
%7n-15$, it is 
%at least reasonable, by the obvious dimensional computation, to expect that, 
%for such a generic irregular surface, $\sQ_{7,n}$ is a direct summand of 
%$H^{0}(\mP^{n},\sI_{S/\mP^{n}}(2))$. Then Theorem C applies to give 
%generic Torelli theorem for such irregular surfaces of general type 
%(with Albanese map of degree $1$ and irregularity $\geq 4$). 

As far as other possible examples are concerned, the class of
minimal irregular surfaces with very ample and primitive canonical bundle should be worthy of study. Indeed 
we remind the reader that the space $\sQ_{k,n}$ of quadrics of the 
projective space $\mP^{n}$ of rank $\leq k$ has dimension 
$k(n-k+1)+\binom{k+1}{2}-1$, and  letting $n=p_{g}(S)-1$, by 
Riemann-Roch theorem the dimension of the vector space of quadrics containing 
$S=Y$ is
$h^{0}(\mP^{n},\sI_{S/\mP^{n}}(2))=
\binom{n+2}{2}-K^{2}_{S}-\chi(\sO_{S})$. Now to apply our theory in the case $n=2$ we have to take $k=2n+3=7$; hence, 
imposing the natural condition $K^{2}_{S}+\chi(\sO_{S})\geq 
7n-15$, it is 
at least reasonable, by the obvious dimensional computation, to expect that, 
for such a generic irregular surface, $\sQ_{7,n}$ is a direct summand of 
$H^{0}(\mP^{n},\sI_{S/\mP^{n}}(2))$. Then Theorem C applies to give 
generic Torelli theorem for such irregular surfaces of general type 
(with Albanese map of degree $1$ and irregularity $\geq 4$). 

Clearly if $X$ is a variety such that there exists a divisor $D$ with $h^0(X,\sO_X(D))\geq 2$ and $h^0(X,\sO_X(K_X-D))\geq 2$ 
then the usual argument of the Petri map for curves gives a quadric of rank 3 or 4 containing the canonical image. It is also easy to construct varieties such that this $D$ exists. Nevertheless our theorem requires that no {\it{adjoint quadrics}} exist, so it can happen that those low rank quadrics obtained by Petri-like  arguments are not adjoint quadrics. We think
that the geometry of $1$-forms and of the quadrics of low rank through the canonical image is an interesting question 
as the problem of finding conditions such that 
$\sQ_{7,n}\cap H^{0}(\mP^{n},\sI_{S/\mP^{n}}(2))=\{0\}$.

\end{document}